\documentclass[11pt]{amsart}
\usepackage{verbatim, graphicx, rotating, rotfloat, lscape}
\usepackage{morefloats}

\usepackage{url}
\usepackage{hyperref}
\usepackage{amssymb}
\usepackage{multirow}

\input xy
\xyoption{all}
\input epsf

\newlength{\tabwidth}
\newlength{\tabheight}
\newlength{\tabrule}
\newlength{\tabwidthx}
\newlength{\tabheightx}

\def\gentabbox#1#2#3#4{\vbox to \tabheight{\setlength{\tabrule}{#3}%
  \setlength{\tabwidthx}{#1\tabwidth}\addtolength{\tabwidthx}{\tabrule}%

\setlength{\tabheightx}{#2\tabheight}\addtolength{\tabheightx}{-\tabheight}%
  \hbox to #1\tabwidth{%
    \hspace{-0.5\tabrule}\rule{\tabrule}{#2\tabheight}\hspace{-\tabrule}%
    \vbox to #2\tabheight{\hsize=\tabwidthx%
      \vspace{-0.5\tabrule}\hrule width\tabwidthx height\tabrule%
      \vspace{-0.5\tabrule}\vfil%
      \hbox to \tabwidthx{\hss#4\hss}%
        \vfil\vspace{-0.5\tabrule}%
      \hrule width\tabwidthx height\tabrule\vspace{-0.5\tabrule}}%
    \hspace{-\tabrule}\rule{\tabrule}{#2\tabheight}\hspace{-0.5\tabrule}}%
  \vspace{-\tabheightx}}}
\def\genblankbox#1#2{\vbox to \tabheight{\vfil\hbox to
#1\tabwidth{\hfil}}}

\newcommand{\xx}{\mathbf x}
\newcommand{\yy}{\mathbf y}

\newcommand{\excise}[1]{}

\newcommand{\field}{\mathbb}
\newcommand{\liealgebra}{\mathfrak}
\newcommand{\la}{\liealgebra}

\newcommand{\C}{{\field C}}

\newcommand{\N}{{\mathbb N}}

\renewcommand{\b}{\liealgebra b}

\newcommand{\n}{{\la n}}

\newcommand{\ga}{\alpha}

\newcommand{\wt}{\widetilde}

\newtheorem{prop}{Proposition}[section]

\newtheorem{theorem}[prop]{Theorem}

\newtheorem{fact}{Fact}

\newtheorem{conjecture}[prop]{Conjecture}

\theoremstyle{definition}

\newtheorem{remark}[prop]{Remark}

\numberwithin{subcase}{case}
\numberwithin{subsubcase}{subcase}
\numberwithin{subsubsubcase}{subsubcase}

\newtheorem{definition}[prop]{Definition}

\newcommand{\frs}{\mathfrak{s}}
\newcommand{\frt}{\mathfrak{t}}

\begin{document}
\title[K-orbit closures on G/B as degeneracy loci]{K-orbit closures on G/B as universal degeneracy loci for flagged vector bundles splitting as direct sums}

\author{Benjamin J. Wyser}
\date{\today}


\begin{abstract}
We use equivariant localization and divided difference operators to determine formulas for the torus-equivariant
fundamental cohomology classes of $K$-orbit closures on the flag variety $G/B$ for various symmetric pairs $(G,K)$.
We describe an interpretation of these formulas as representing the classes of particular types of degeneracy loci
when evaluated at certain Chern classes.  For the type $A$ pair $(SL(p+q,\C),S(GL(p,\C) \times GL(q,\C)))$, such
degeneracy loci are described explicitly, relative to a rank $p+q$ vector bundle $V$ on a smooth complex variety $X$
equipped with a flag of subbundles and a splitting of $V$ as a direct sum of subbundles of ranks $p$ and $q$.
We conjecture similarly explicit descriptions of the degeneracy loci for all cases in types $B$ and $C$.
\end{abstract}

\maketitle

Suppose that $G$ is a complex reductive group of classical type, and that $K=G^{\theta}$ is the subgroup fixed by an involution $\theta$ of $G$.  $K$ is referred to as a \textit{symmetric subgroup}.  $K$ acts on the flag variety $G/B$ with finitely many orbits \cite{Matsuki-79}, and the geometry of these orbits and their closures plays an important role in the theory of Harish-Chandra modules for a certain real form of the group $G$.  For this reason, the geometry of $K$-orbits and their closures have been studied extensively, primarily in representation-theoretic contexts.

In \cite{Wyser-13-TG}, the $K$-orbit closures for the symmetric pairs $(G,K)=(GL(n,\C),O(n,\C))$, 
$(SL(n,\C),SO(n,\C))$, and $(SL(2n,\C),Sp(2n,\C))$ were studied from the perspective of torus-equivariant
geometry.  In the current paper, we carry out a similar program of study for the remaining symmetric
pairs $(G,K)$ with $G$ a classical simple group, up to finite covers.  The specific pairs that we study can be found
in Table \ref{tab:all-pairs}.

The study of \cite{Wyser-13-TG}, as well as that of this paper, was motivated by earlier work of W. Fulton
\cite{Fulton-92,Fulton-96_1,Fulton-96_2} which realized Schubert varieties as universal degeneracy loci for maps of
flagged vector bundles, and by connections between that work and the equivariant cohomology of the flag variety,
elucidated by W. Graham in \cite{Graham-97}.  $K$-orbit closures are, in a sense, generalizations of Schubert
varieties, and so it is natural to try to fit these more general objects into a framework similar to that described
in the aforementioned works.

To this end, in \cite{Wyser-13-TG}, the following program is carried out for the pairs $(GL(n,\C),O(n,\C))$, $(SL(n,\C), SO(n,\C))$ and $(SL(2n,\C),Sp(2n,\C))$:
\begin{enumerate}
	\item Determine formulas for the $S$-equivariant cohomology classes of the closed $K$-orbits using equivariant localization, together with the self-intersection formula.  (Here, $S$ is a maximal torus of $K$ contained in a $\theta$-stable maximal torus $T$ of $G$.)
	\item Using such formulas as a starting point, describe the weak order on $K \backslash G/B$ combinatorially, and outline how divided difference calculations can give formulas for the equivariant classes of the remaining orbit closures.
	\item Describe the $K$-orbit closures explicitly as sets of flags, and using this description,
realize the $K$-orbit closures as universal degeneracy loci of a certain type, involving a vector bundle over a smooth
complex variety $X$ equipped with a single flag of subbundles and a certain additional structure depending on $K$.
\end{enumerate}

In the present paper, we carry out this program for the remaining symmetric pairs $(G,K)$ of Table \ref{tab:all-pairs}.  Step (1) is
carried out completely for all pairs $(G,K)$ listed above in the form of Theorem \ref{thm:formulas}, the main result of this
paper.  For step (2), there is little to do, as combinatorial models for $K \backslash G/B$, as well as their weak
orders, are already understood \cite{Matsuki-Oshima-90}.  Indeed, the only issue for us not addressed by \emph{loc. cit.}
is distinguishing between solid and dashed edges in our
weak order Hasse diagrams.  This
concerns the question of whether or not to divide by $2$ when performing a divided difference calculation.  This
matter has been fully addressed in existing literature in some cases, whereas the remaining cases are all addressed in 
\cite{Wyser-Thesis}.  See Section \ref{sec:solid-dashed} for details.

As for step (3), in the cases considered in \cite{Wyser-13-TG}, the ``additional structure" possessed by the vector
bundle is a non-degenerate symmetric or skew-symmetric bilinear form taking values in the trivial bundle.  The
associated degeneracy loci are defined by imposing conditions on the rank of the form when it is restricted to the
fibers of various components of the flag.  These conditions come from explicit knowledge of the $K$-orbit
closures as sets of flags in each case.

In this paper, the relevant additional structure is a splitting of the vector bundle as a direct
sum of two subbundles.  Thus the degeneracy loci corresponding to the $K$-orbit closures in these cases can all be
described roughly as follows:  We are given a complex vector bundle $V$ over a smooth complex variety $X$, a single
flag $F_{\bullet}$ of subbundles of $V$, and a splitting of $V$ as a direct sum $V' \oplus V''$ of subbundles.
(In types $BCD$, the bundle $V$ is equipped with a symmetric or skew-symmetric bilinear form, the flag $F_{\bullet}$
is isotropic/Lagrangian with respect to the form, and the summands $V'$ and $V''$ are required to satisfy further
properties with respect to the form.)  The degeneracy loci are then defined by imposing conditions on the relative
position of the fibers of the flag and the two summands.  These conditions are encoded in ``clans" (which are
character strings consisting of $+$'s, $-$'s, and natural numbers subject to some further conditions) which
parametrize the $K$-orbit closures.  Their precise nature depends again upon explicit set-theoretic descriptions of
the orbit closures, which have not appeared in the literature in any of our cases.  We give such an explicit
description for case (1) of Table \ref{tab:all-pairs} as Theorem \ref{thm:orbit-closures}.

We do not know of similarly explicit descriptions of orbit closures for the remaining cases.  This leaves item (3)
partially unaddressed.  However, as the $K$-orbits in the remaining cases are closely related to those in case (1)
(in a sense made precise in Section \ref{sec:other-cases}), Theorem \ref{thm:orbit-closures} does at least suggest a naive guess at such
a description, which we conjecture is correct in types $BC$, i.e. cases (2)-(4) of Table \ref{tab:all-pairs} (cf. Conjecture
\ref{conj:bruhat-order-other-types}).  Alas, this naive guess is incorrect for all three type $D$ pairs, as we note in
Fact \ref{fact:type-d-conj-false}.

The paper is organized as follows:  In Section~1, we recall various preliminary facts on $K$-orbits and weak order, as well as the known combinatorial models for the various orbit sets.  In Section~2, we start by reviewing some basic facts
on equivariant cohomology and the localization theorem, then use these facts to prove formulas for the classes of
closed $K$-orbits in the various cases.  These are summarized in Theorem \ref{thm:formulas}, the main result of the
paper.  Finally, in Section~3, we connect these formulas to Chern class formulas for degeneracy loci of the type
loosely described above.

A number of the results presented herein were part of the author's PhD thesis, written at the University of Georgia
under the direction of his research advisor, William A. Graham.  The author thanks Professor Graham wholeheartedly
for his help in conceiving that project, as well as for his great generosity with his time and expertise throughout.
The author also thanks Michel Brion for helpful remarks and advice, and an anonymous referee for many helpful suggestions
which greatly improved the exposition of an earlier version of the papers.

\section{Preliminaries}
\subsection{Notation}\label{sec:notation}
We denote by $I_n$ the $n \times n$ identity matrix, and by $J_n$ the $n \times n$ matrix with $1$'s on the
antidiagonal and $0$'s elsewhere, i.e. the matrix $(e_{i,j}) = \delta_{i,n+1-j}$.  If $n = p+q$, then $I_{p,q}$ will
denote the $n \times n$ diagonal matrix having $p$ $1$'s followed by $q$ $-1$'s on the diagonal.  If $a+b+c=n$, then
$I_{a,b,c}$ will denote the $n \times n$ diagonal matrix having $a$ $1$'s, followed by $b$ $-1$'s, followed by $c$
$1$'s, on the diagonal.  $J_{n,n}$ shall denote the block matrix which has $J_n$ in the upper-right block, $-J_n$ in
the lower-left block, and $0$'s elsewhere.  That is,
\[  J_{n,n} := 
\begin{pmatrix}
0 & J_n \\
-J_n & 0 \end{pmatrix}. \]

Ordinary permutations will typically be written in one-line notation, i.e. the permutation in $S_4$ which sends $1$ to
$2$, $2$ to $3$, $3$ to $1$, and $4$ to $4$, will be indicated by $2314$.  We will deal also with signed permutations,
which will typically be written in one-line notation with bars over some of the numbers to indicate where the 
permutation takes negative values.  Thus the signed permutation of $\{1,\hdots,4\}$ which sends $1$ to $2$,
$2$ to $-3$, $3$ to $1$, and $4$ to $-4$ will be indicated by $2 \overline{3} 1 \overline{4}$.

Unless stated otherwise, $H^*(-)$ (resp. $H_S^*(-)$) shall always mean ($S$-equivariant) cohomology with $\C$-coefficients.

We denote the Lie algebra of an algebraic group by its corresponding Gothic letter, so $\text{Lie}(S)$ is denoted by
$\frs$, $\text{Lie}(T)$ by $\frt$, and so on.  If $S \subseteq T$ is an inclusion of algebraic tori, then dual to the inclusion of $\frs \subseteq \frt$ there is a restriction map
$\frt^* \rightarrow \frs^*$, which we denote by $\rho$.
Coordinate functions on $\frs$ will be denoted by capital $Y$ variables, while those on $\frt$ will be denoted by
capital $X$ variables.

$K \backslash G/B$ should always be taken to mean the set of $K$-orbits on
$G/B$, unless explicitly stated otherwise.  (This is as opposed to $B$-orbits on $K \backslash G$, or $B \times K$-orbits
on $G$.)

If $Y$ is a $K$-orbit closure, and $\ga$ a simple root, denote by $s_{\ga} \cdot Y$ the $K$-orbit closure
$\pi_{\ga}^{-1}(\pi_{\ga}(Y))$, where $\pi_{\ga}: G/B \rightarrow G/P_{\ga}$ is the standard projection.
($P_{\ga}$ denotes the standard parabolic of type $\ga$, i.e. $P_{\ga} = B \cup Bs_{\ga}B$.)  Since $\pi_{\ga}$ is a 
$\mathbb{P}^1$-bundle, $s_{\ga} \cdot Y$ is either equal to $Y$, or it is another $K$-orbit closure of dimension
$\dim(Y)+1$.

Given $w \in W$ with $\ell(w) = l$, let $w=s_{\ga_1} \hdots s_{\ga_l}$ be a reduced decomposition of $w$.  Then define
\[ w \cdot Y = s_{\ga_1} \cdot (s_{\ga_2} \cdot \hdots \cdot (s_{\ga_l} \cdot Y) \hdots ). \]
The resulting $K$-orbit closure is independent of the choice of reduced expression for $w$ \cite{Richardson-Springer-90}.

\begin{definition}\label{def:weak-order}
The \textbf{weak order} on the set of $K$-orbit closures is defined by $Y \leq Y'$ if and only if
$Y' = w \cdot Y$ for some $w \in W$.
\end{definition}

In our Hasse diagrams depicting the weak order, if $Y' = s_{\ga} \cdot Y$, then $Y$ is connected to $Y'$ by a solid
edge labelled by $\ga$ if $\pi_{\ga}|_Y: Y \rightarrow \pi_{\ga}(Y)$ is birational.  We use a dashed edge if instead this map is $2$-to-$1$.  These are the only two possibilities.

\subsection{The specific cases}\label{sec:specific-cases}
With notation as defined in Section \ref{sec:notation}, we describe the particular realizations of the various
symmetric pairs $(G,K)$ that we have in mind.  We start with the lone type $A$ case $(SL(n,\C),S(GL(p,\C) \times GL(q,\C)))$
($n=p+q$).  In some sense, this case is the most important, as our understanding of the combinatorics of the remaining
cases depends heavily upon our understanding of this one.  We define $K$ to be $SL(n,\C)^{\theta}$, where $\theta$
is the involution $\text{int}(I_{p,q})$, and where $\text{int}(g)$ denotes the inner automorphism ``conjugation by $g$''.
Then $K$ is embedded in $G$ in the obvious way:
\[ K = \left\{
\left[
\begin{array}{cc}
K_{11} & 0 \\
0 & K_{22} \end{array}
\right] \in SL(p+q,\C)
\ \middle\vert \ 
\begin{array}{c}
K_{11} \in GL(p,\C) \\
K_{22} \in GL(q,\C) \end{array}
\right\} \cong S(GL(p,\C) \times GL(q,\C)). \]

To describe our realizations of the remaining pairs, we first specify a realization of the ambient group $G$ in each case.

For $G=SO(2n+1,\C)$ of type $B$, we let $G$ be the isometry group of the symmetric form defined by
$\left\langle e_i,e_j \right\rangle = \delta_{i,2n+1-j}$.  

In matrix terms,
\[ G = \{M \in SL(2n+1,\C) \mid MJ_{2n+1}M^t = J_{2n+1}\}. \]

For $G=Sp(2n,\C)$ of type $C$, 
\[ G = \{M \in SL(2n,\C) \mid MJ_{n,n}M^t = J_{n,n}\}. \]

Finally, for $G=SO(2n,\C)$ of type $D$, in cases (5) and (6) we use the following realization of $G$:
\[ G = \{M \in SL(2n,\C) \mid MJ_{2n}M^t = J_{2n}\}. \]

In case (7), on the other hand, we use a different (perhaps more typical) realization of $SO(2n,\C)$:
\[ G = \{M \in SL(2n,\C) \mid MM^t = I_{2n}\}. \]

We specify the various groups $K$ by describing their involutions in the following table.  In all cases, $n=p+q$.

\begin{table}[h]
\begin{tabular}{|l|l|l|}
		\hline
		Case number & $(G,K)$ & Involution $\theta$ \\ \hline
		(1) &  $(SL(n,\C),S(GL(p,\C) \times GL(q,\C))$ & $\text{int}(I_{p,q})$ \\ \hline
		(2) & $(SO(2n+1,\C),S(O(2p,\C) \times O(2q+1,\C)))$ & $\text{int}(I_{p,2q+1,p})$ \\ \hline
		(3) & $(Sp(2n,\C),Sp(2p,\C) \times Sp(2q,\C))$ & $\text{int}(I_{p,2q,p})$ \\ \hline
		(4) & $(Sp(2n,\C),GL(n,\C))$ & $\text{int}(iI_{n,n})$ \\ \hline
		(5) & $(SO(2n,\C),S(O(2p,\C) \times O(2q,\C)))$ & $\text{int}(I_{p,2q,p})$ \\ \hline
		(6) & $(SO(2n,\C),GL(n,\C))$ & $\text{int}(iI_{n,n})$ \\ \hline
		(7) & $(SO(2n,\C),S(O(2p+1,\C) \times O(2q-1,\C)))$ & $g \mapsto I_{2p+1,2q-1}gI_{2p+1,2q-1}$ \\
		\hline
\end{tabular}
\caption{Symmetric pairs and corresponding involutions}\label{tab:all-pairs}
\end{table}

In all cases but (7), $\theta$ is an inner involution; in case (7), it is not.  Note also that in each of  (2)-(7), it is the case that
$K=G \cap K'$, where $K'$ is $S(GL(p,\C) \times GL(q,\C))$ for some $p$ and $q$ (perhaps twisted by an outer
automorphism of $SL(p+q,\C)$).  As examples, in case (2), $K' = S(GL(2p,\C) \times GL(2q+1,\C))$, while in case (4),
$K'=S(GL(n,\C) \times GL(n,\C))$.

\subsubsection{Choices of maximal tori, Borel subgroups, and root data}
We now make explicit choices of maximal tori and Borel subgroups of both $G$ and $K$, and describe the corresponding
embeddings of root systems and Weyl groups.  We always choose $T \subseteq B$ to be a $\theta$-stable maximal torus of
$G$ contained in a $\theta$-stable Borel subgroup of $G$.  It is known abstractly that such pairs $T \subseteq B$
always exist, cf.  \cite{Richardson-Springer-90} and references therein.  In cases (1)-(6), we choose
$T$ to be the diagonal elements of $G$, and $B$ to be the lower triangular elements of $G$.  We take this choice of
$B$ because we want $B$ to correspond to the negative roots $-\Phi^+$, with $\Phi^+ \subseteq \Phi(G,T)$ one of the following
``standard'' positive root systems:
\begin{itemize}
 \item Type $A$:  $\{X_i - X_j \mid i < j\}$;
 \item Type $B$:  $\{X_i \pm X_j \mid i < j\} \cup \{X_i \mid i = 1,\hdots,n\}$;
 \item Type $C$:  $\{X_i \pm X_j \mid i < j\} \cup \{2X_i \mid i = 1,\hdots,n\}$;
 \item Type $D$:  $\{X_i \pm X_j \mid i < j\}$.
\end{itemize}

In case (7), the torus $T$ we consider is the one such that $\frt$ consists of matrices
of the following form:
\[
	\begin{pmatrix}
		\begin{tabular}[t]{c|c|}
			$0$ & $a_1$ \\ \hline
			$-a_1$ & $0$ \\ \hline
		\end{tabular}
		& & & \multirow{2}{*}{\Large 0} \\
		& 
		\begin{tabular}[b]{|c|c|}
			\hline
			$0$ & $a_2$ \\ \hline
			$-a_2$ & $0$ \\ \hline
		\end{tabular}
		& & \\
		\multirow{2}{*}{\Large 0} & & \ddots & 
		\\
		& & & 
		\begin{tabular}{|c|c}
			\hline
			$0$ & $a_n$ \\ \hline
			$-a_n$ & $0$
		\end{tabular}
	\end{pmatrix}
\]
$B$ is the Borel which contains $T$ and which corresponds to the roots $-\Phi^+$, where $\Phi^+$ is again taken to be
the standard positive system defined above.

In all cases, our choices of $B$ and $T$ are such that $B_K = B \cap K$ is a Borel subgroup of $K$, and such that
$S = T \cap K$ is a maximal torus of $K$.  In cases (1)-(6), we actually have $S=T$.  Even so, we will denote the
torus of $K$ by $S$, and coordinate functions on $\frs$ by variables $Y_i$, with coordinate functions on $\frt$
denoted by variables $X_i$.  In cases (1)-(6), the restriction map $\rho$ mentioned in Section \ref{sec:notation} is
simply given by $X_i \mapsto Y_i$, and we often omit it from the notation.

In case (7), $\text{rank}(K) = n-1$, and $S$ is the subtorus of $T$ such that $\frs$ consists of matrices of the form
\[
	\begin{pmatrix}
		\begin{tabular}{c|c|}
			$0$ & $a_1$ \\ \hline
			$-a_1$ & $0$ \\ \hline
		\end{tabular}
		& & & & & & \\
		& \ddots & & & & \multirow{2}{*}{\Huge 0} & \\
		& & 
		\begin{tabular}{|c|c|}
			\hline
			$0$ & $a_p$ \\ \hline
			$-a_p$ & $0$ \\ \hline
		\end{tabular}
		& & & & \\
		& & & 
		\begin{tabular}{|c|c|}
			\hline
			$0$ & $0$ \\ \hline
			$0$ & $0$ \\ \hline
		\end{tabular}
		& & & \\
		& & & &
		\begin{tabular}{|c|c|}
			\hline
			$0$ & $a_{p+2}$ \\ \hline
			$-a_{p+2}$ & $0$ \\ \hline
		\end{tabular}
		& & \\
		& \multirow{2}{*}{\Huge 0} & & & & \ddots & \\
		& & & & & &
		\begin{tabular}{|c|c}
			\hline
			$0$ & $a_n$ \\ \hline
			$-a_n$ & $0$
		\end{tabular}
	\end{pmatrix}
\]

For this case, the map $\rho$ is given by $X_i \mapsto Y_i$ for $i \neq p+1$, and $X_{p+1} \mapsto 0$.

With all these choices made, the root system of $K$ in each case is given in Table \ref{tab:roots-of-k}, which appears
as part of the proof of Theorem \ref{thm:formulas}, for ease of reference by those who would carefully read
that proof.

\subsubsection{Weyl groups}
Finally, we also describe the Weyl group $W$ for $G$ in each of our cases, as well as the embedding of $W_K$, the Weyl
group for $K$, into $W$.  Note that since $S=T$ in cases (1)-(6), it is obvious that the map $W_K := N_K(S)/S
\rightarrow N_G(T)/T =: W$ is an embedding $W_K \subseteq W$.  In cases such as (7), where $S \subsetneq T$, this is
not entirely obvious, but is still true, as explained in \cite[Section 1]{Wyser-13-TG}.

In type $A$, $W \cong S_n$, as usual.  In types $B$ and $C$, $W$ is the group of signed permutations of
$\{1,\hdots,n\}$.  These are bijections $\sigma$ on $\{\pm 1,\hdots,\pm n\}$ such that $\sigma(-i) = -\sigma(i)$ for
all $i$.  In type $D$, $W$ is the group of all signed permutations on $\{1,\hdots,n\}$ which change an even number of
signs.  In each case, $W$ acts on the coordinate functions $X_i$ by permutation of the indices together with sign
changes.

In the various cases, $W_K$ is embedded as follows:
\begin{enumerate}
	\item $W_K \cong S_p \times S_q$, permutations which act separately on the sets $\{1,\hdots,p\}$ and $\{p+1,\hdots,n\}$;
	\item $W_K$ consists of signed permutations which act separately on the sets $\{\pm 1,\hdots,\pm p\}$ and $\{ \pm (p+1),\hdots,\pm n\}$;
	\item Same as case (2);
	\item $W_K \cong S_n$, embedded in $W$ as the signed permutations of $\{\pm 1,\hdots,\pm n\}$ which change no signs.
	\item Same as cases (2) and (3);
	\item Same as case (4);
	\item $W_K$ consists of signed permutations of $\{1,\hdots,n\}$ which act separately on $\{\pm 1,\hdots, \pm p\}$ and $\{\pm (p+2),\hdots,\pm n\}$, changing any number of signs on each set,
and which either fix $p+1$ or send it to its negative, whichever causes the 
resulting signed permutation to have an even number of sign changes.
\end{enumerate}

Note that in cases (2), (5), and (7), $K$ is disconnected.  In each case, we have described $W_K = N_K(S)/S$, which is
\textit{a priori} different from the Weyl group for the Lie algebra of $K$.  The latter is $W_{K^0}$, where $K^0$ is
the identity component of $K$.  In cases (2) and (5), $W_{K^0}$ is an index $2$ subgroup of $W_K$.  In case (2), it
consists of those elements of $W_K$ which change an even number of signs on $\{\pm 1,\hdots,\pm p\}$.  In case (5), it
consists of those elements of $W_K$ which change an even number of signs on both the sets $\{\pm 1,\hdots,\pm p\}$ and
$\{\pm (p+1),\hdots,\pm n\}$.  By contrast, in case (7) we actually have $W_K = W_{K^0}$.

\subsection{Parametrization of $K$-orbits by Clans}\label{sec:param-clans}
Combinatorial models for $K \backslash G/B$, as well as their weak orders, are described for all of the cases of this paper in \cite{Matsuki-Oshima-90}, with some of the cases being treated in more detail in \cite{Yamamoto-97}.  We recall these known parametrizations.

\subsubsection{$(G,K) = (SL(p+q,\C),S(GL(p,\C) \times GL(q,\C)))$}\label{sec:type-a-case}
As mentioned above, the most important pair we consider is $(G,K)=(SL(p+q,\C),S(GL(p,\C) \times GL(q,\C)))$, since the combinatorics of all of
the other cases are derived from this one, in a sense made precise in the next section.  Thus we recall the known results of this case first.
Appropriate references are \cite{Matsuki-Oshima-90,Yamamoto-97}.

For this pair, the $K$-orbits are parametrized by what are called ``clans"  or, if we wish to emphasize the role of
$p$ and $q$, ``$(p,q)$-clans".
  
\begin{definition}
A \textbf{$(p,q)$-clan} is an involution (i.e. an element of order $2$) in $S_{p+q}$ with each fixed point decorated by either a $+$ or a $-$
sign, in such a way that the number of $+$ fixed points minus the number of $-$ fixed points is $p-q$.  (If
$p < q$, then there should be $q-p$ more $-$ signs than $+$ signs.)

A clan is depicted by a string $c_1 \hdots c_n$ of $p+q$ characters.  Where the underlying involution fixes
$i$, $c_i$ is either a $+$ or a $-$, as appropriate.  Where the involution interchanges $i$ and $j$, 
$c_i=c_j\in \N$ is a matching pair of natural numbers.  (A different natural number is used for each 
pair of indices exchanged by the involution, and these character strings are considered equivalent up to permutation of the natural numbers; e.g. $1122$ and $2211$ are the same, since they encode the same involution $2143$.)
\end{definition}

As an example, suppose $n = 4$, and $p = q = 2$.  Then we must consider all clans of length $4$ where the number of $+$'s and the number of $-$'s is the same (since $p-q = 0$).  There are $21$ of these, and they are as follows:  
	\[ ++--; +-+-; +--+; -++-; -+-+; --++; \]
	\[ 11+-; 11-+; 1+1-; 1-1+; 1+-1; 1-+1; \]
	\[ +11-; -11+; +1-1; -1+1; +-11; -+11; \]
	\[ 1122; 1212; 1221 \]

We give a combinatorial description of the weak order (cf. Definition \ref{def:weak-order}) in this case, relative to this parametrization by clans.  The minimal orbits are those whose indexing clans consist only of signs.  From there, we need only define $s_i \cdot \gamma$, where
$s_i$ is the simple transposition in $S_n$ which interchanges $i$ and $i+1$, and where $\gamma$ is an arbitrary
$(p,q)$-clan.

If $\gamma=c_1 \hdots c_n$, then $s_i \cdot \gamma \neq \gamma$ if and only if one of the following holds:
\begin{enumerate}
	\item $c_i$ and $c_{i+1}$ are unequal natural numbers, and the mate of $c_i$ is to the left of the mate of $c_{i+1}$;
	\item $c_i$ is a sign, $c_{i+1}$ is a natural number, and the mate of $c_{i+1}$ is to the right of $c_{i+1}$;
	\item $c_i$ is a natural number, $c_{i+1}$ is a sign, and the mate of $c_i$ is to the left of $c_i$; or
	\item $c_i$ and $c_{i+1}$ are opposite signs.
\end{enumerate}

For example, taking $p = 3, q = 2$, and letting $i=2$, $112+2$ satisfies the first condition; $+-11+$ satisfies the second; $11+-+$ satisfies the third; and $+-++-$ satisfies the fourth.  Note that $1+122$ satisfies none of the conditions, since the mate of the $1$ in the 3rd slot occurs to its \textit{left}, rather than to its right.

If $\gamma' = s_i \cdot \gamma \neq \gamma$, then in the first three cases, $\gamma'$ is obtained from $\gamma$ by interchanging $c_i$ and $c_{i+1}$.  In the fourth case, $\gamma'$ is obtained from $\gamma$ by replacing the opposite signs $c_i$ and $c_{i+1}$ by a pair of equal natural numbers.  So, for the examples above, we have
\begin{itemize}
	\item $s_2 \cdot 112+2 = 121+2$;
	\item $s_2 \cdot +-11+ = +1-1+$;
	\item $s_2 \cdot 11+-+ = 1+1-+$;
	\item $s_2 \cdot +-++- = +11+-$.
\end{itemize}

On the other hand, $s_2 \cdot 1+122 = 1+122$.

\subsubsection{Other pairs}\label{sec:other-cases}
As noted in Section \ref{sec:specific-cases}, for the other symmetric pairs $(G,K)$ considered in this paper (those in types $BCD$), $K$ can be
realized as $G \cap K'$ with $K' \cong S(GL(p,\C) \times GL(q,\C)) \subseteq G' = SL(p+q,\C)$ for some $p$ and $q$.  The
flag variety $X$ for $G$ naturally embeds in the flag variety $X'$ for $G'$, and so the intersection of a $K'$-orbit
on $X'$ with $X$, if non-empty, is stable under $K$ and hence is \textit{a priori} a union of $K$-orbits.

In general, such an intersection need not be a single $K$-orbit.  It could in principle be a union of
multiple $K$-orbits, and indeed this can happen.  Whether it happens depends upon the chosen
representative of the isogeny class of $G$, which in turn can affect the connectedness of $K$.  It turns out that in
each case we consider, it is possible to choose $G$ (and the corresponding $K$) such that the
intersection of a $K'$-orbit on $X'$ with $X$ is always a \textit{single} $K$-orbit.  These particular choices of
$(G,K)$ are the ones appearing in Table \ref{tab:all-pairs}.

The upshot is that in each of our cases outside of type $A$, the set of $K$-orbits can be parametrized by a
subset of the $(p,q)$-clans (for the appropriate $p,q$) possessing some additional combinatorial properties which
amount to the corresponding $K'$-orbit on $X'$ meeting the smaller flag variety $X$
non-trivially.  These combinatorial properties always involve one of the following two symmetry conditions.

\begin{definition}\label{def:symmetric-clan}
We say that $\gamma=c_1 \hdots c_n$ is \textbf{symmetric} if the clan $c_n \hdots c_1$ obtained from $\gamma$ by
reversing its characters is equal to $\gamma$ as a clan.  Explicitly, we require
\begin{enumerate}
	\item If $c_i$ is a sign, then $c_{n+1-i}$ is the same sign.
	\item If $c_i$ is a number, then $c_{n+1-i}$ is also a number, and if $c_{n+1-i} = c_j$, then $c_{n+1-j} = c_i$.
\end{enumerate}
\end{definition}

\begin{definition}\label{def:skew-symmetric-clan}
We say that $\gamma=c_1 \hdots c_n$ is \textbf{skew-symmetric} if the clan $c_n \hdots c_1$ is the ``negative" of
$\gamma$, meaning it is the same clan, except with all signs changed.  Specifically,
\begin{enumerate}
	\item If $c_i$ is a sign, then $c_{n+1-i}$ is the opposite sign.
	\item If $c_i$ is a number, then $c_{n+1-i}$ is also a number, and if $c_{n+1-i} = c_j$, then $c_{n+1-j} = c_i$.
\end{enumerate}
\end{definition}

Note that condition (2) of each of the above definitions allows for the possibility that $c_i = c_{n+1-i}$.  However,
this is not necessary for a clan to be symmetric or skew-symmetric.  Indeed, the $(2,2)$-clan $1212$ is symmetric
(and also skew-symmetric), since its reverse $2121$ is the same clan, but there are no matching natural numbers in
positions $(i,n+1-i)$ for any $i$.

The $K$-orbits in our remaining examples are parametrized as follows.  Note that the numbering here starts from (2)
in order to make this list correspond to that given in Table \ref{tab:all-pairs}, as well as in the statement of Theorem
\ref{thm:formulas}, in the form of Table \ref{tab:formulas}.

\begin{enumerate}
\setcounter{enumi}{1}
	\item $(SO(2n+1,\C),S(O(2p,\C) \times O(2q+1,\C)))$:  Symmetric $(2p,2q+1)$-clans;
	\item $(Sp(2n,\C),Sp(2p,\C) \times Sp(2q,\C))$:  Symmetric $(2p,2q)$-clans $\gamma=c_1 \hdots c_{2n}$ such that $c_i \neq c_{2n+1-i}$ whenever $c_i \in \N$;
	\item $(Sp(2n,\C),GL(n,\C))$:  Skew-symmetric $(n,n)$-clans;
	\item $(SO(2n,\C), S(O(2p,\C) \times O(2q,\C)))$:  Symmetric $(2p,2q)$-clans;
	\item $(SO(2n,\C), GL(n,\C))$:  Skew-symmetric $(n,n)$-clans $\gamma=c_1 \hdots c_{2n}$ such that $c_i \neq c_{2n+1-i}$ whenever $c_i \in \N$, and such that among $c_1 \hdots c_n$, the total number of $-$ signs and pairs of equal natural numbers is even;
	\item $(SO(2n,\C), S(O(2p+1,\C) \times O(2q-1,\C)))$:  Symmetric $(2p+1,2q-1)$-clans.
\end{enumerate}

These parametrizations, along with a description of the corresponding weak orders, appear in \cite{Matsuki-Oshima-90}.
We remark that no proofs of the correctness of the parametrizations are given in \emph{loc. cit.}.  Proofs are
given in \cite{Yamamoto-97} for two of the above cases (as well as for the pair $(SL(p+q,\C),S(GL(p,\C) \times GL(q,\C)))$.
Proofs of the correctness of all of these parametrizations appear in \cite[Appendix A]{Wyser-Thesis}.

\subsection{Recursive computation starting from closed orbits}\label{ssec:other_orbits}
We recall our general method for computing representatives, reminding the reader of the geometric justification for it.
The idea is to start from explicit representatives of the classes of closed $K$-orbits, and determine formulas for
classes of the remaining $K$-orbit closures using divided difference operators.  We quickly review how this works,
referring the reader to \cite[Section 1.4]{Wyser-13-TG} for more details.

Suppose that $Y$ and $Y'$ are two $K$-orbit closures with $Y' = s_{\ga} \cdot Y$.  Then in the Hasse diagram, $Y$ is connected to $Y'$ by an edge labelled $\ga$.  This edge is either solid or dashed, depending on whether $\pi_{\ga}|Y$ has degree $1$ or $2$, respectively.  If it is solid, we have
\[ [Y'] = \partial_{\ga}[Y], \]
both in ordinary and $S$-equivariant cohomology.  If it is dashed, we have
\[ [Y'] = \frac{1}{2} \partial_{\ga}[Y]. \]
Here $\partial_{\ga}$ denotes the \textbf{divided difference operator} on $H_S^*(X)$
corresponding to $\ga$, defined by
\[ \partial_{\ga}(f) = \frac{f - s_{\ga}(f)}{\ga}. \]

Given the above, the following result \cite[Theorem 4.6]{Richardson-Springer-90} says that if we have explicit formulas
for the classes of closed orbits, then in principle we can compute formulas for the classes of all other orbit
closures via successive divided difference calculations.

\begin{theorem}
  Let $Y$ be a $K$-orbit closure on $G/B$.  There exists a closed orbit $Y_0$ and some $w \in W$ such that
$w \cdot Y_0 = Y$.
\end{theorem}

Thus what we need in each case is a formula for the class of each closed orbit, along with a concrete understanding of the combinatorics of the orbit set and its weak order, including which edges are solid and which are dashed.  The orbit sets were described in Section \ref{sec:param-clans}, and as indicated there, their weak orders are all understood combinatorially.  Thus we only need address the matter of solid versus dashed edges.

\subsection{Solid versus dashed edges}\label{sec:solid-dashed}
Though the orbit sets and their weak orders are described explicitly in \cite{Matsuki-Oshima-90}, the matter of which edges are solid and which are dashed is not fully addressed in any particular reference.  However, this is straightforward to determine by combining the combinatorics described in \cite{Matsuki-Oshima-90} with known results from other references, specifically\cite{Richardson-Springer-90,Richardson-Springer-92,Vogan-83}.  In what follows, we will use some terms from these latter references without defining them carefully; the interested reader can consult the references for definitions.

There is an edge labelled by $s_{\ga}$ connecting $Y:=\overline{Q}$ to $Y':=\overline{Q'}$ in the weak order
graph only if the simple root $\ga$ is \textbf{complex} or \textbf{non-compact imaginary} for the orbit $Q$.  Non-compact
imaginary roots are furthermore subdivided into two types, referred to as \textbf{type I} and \textbf{type II}.
In the event that the root $\ga$ is complex or non-compact imaginary of type I, the corresponding edge connecting $Y$ to $Y'$ is solid.  Only in the non-compact imaginary type II case is it dashed.

In the description of the weak order given in \cite{Matsuki-Oshima-90}, a distinction is made between complex and
non-compact imaginary roots.  However, none is made between non-compact imaginary roots of type I and type II.
These two cases can be easily distinguished, though, by considering the \textbf{cross-action} of the simple
reflection $s_{\ga}$ on the orbit $Q$.  The cross-action of $W$ on $K \backslash G/B$ is defined by
\[ w \times (K \cdot gB/B) = K \cdot gw^{-1}B/B. \]
Then if $s_{\ga} \cdot Y = Y'$, and $\ga$ is non-compact imaginary for $Q$, it is of type I (solid edge) if
$s_{\ga} \times Q \neq Q$, and of type II (dashed edge) if $s_{\ga} \times Q = Q$.

It is known how to compute the cross-action explicitly in our examples.  For instance, for the type $A$ pair 
$(SL(p+q,\C),S(GL(p,\C) \times GL(q,\C)))$, it is computed by the obvious permutation action of $w$ on the characters of
the clan indexing the orbit in question; in particular, if $
gamma=c_1 \hdots c_n$, then the cross-action of a simple reflection $s_i$ on the corresponding $Q_{\gamma}$ gives $Q_{\gamma'}$, where $\gamma'$ is obtained from $\gamma$ by interchanging $c_i$ and $c_{i+1}$.  From \cite{Matsuki-Oshima-90}, we see that a simple root $\ga_i$ is non-compact
imaginary for $Q_{\gamma}$ only in the case where $c_i$ and $c_{i+1}$ are opposite
signs (cf. the description of the weak order given in Section \ref{sec:type-a-case}).  Since the cross-action is to interchange these opposite signs, giving a different clan,
we note that all non-compact imaginary roots are of type I, and so all edges in the weak order graph are solid.

We remark that this has been noted before, for example by M. Brion in \cite{Brion-01}, and again by E.Y. Smirnov in
\cite{Smirnov-08}.  (The latter reference considers $B$-orbits on a product $X$ of two Grassmannians, a compactification of $G/K$, and proves the stronger result that the weak order graph for $B$-orbit closures on $X$ contains no dashed edges.)  Similar reasoning shows that for pairs (3) and (6) in our numbering scheme, the
weak order graph again contains only solid edges.  For pair (6), this is also noted in \cite{Brion-01}.  It can also be recovered from a suitable generalization of the aforementioned result of \cite{Smirnov-08}, proved in \cite{Achinger-Perrin}.

Pairs (2), (4), (5), and (7) can have dashed edges in their weak order graphs, as one can see by examining the
various example graphs given in the Appendix.  In the interest of brevity, we do not list here the specific
combinatorial rules for determining the weak order and the solid/dashed edges for every possible case.
The interested reader can consult \cite{Matsuki-Oshima-90} for the former, and \cite{Wyser-Thesis} for the latter.

\section{Formulas for the Equivariant Classes of Closed Orbits}
\subsection{Background and the basic method}\label{sec:eqvt-cohomology-background}
Before giving our equivariant formulas, we review some basics of equivariant cohomology which support our methods of computation.  Results of this section are stated without proof, as they are fairly standard.  The reader seeking a reference can consult \cite{Wyser-13-TG} for an expository treatment.

We work in equivariant cohomology with respect to the action of a maximal torus $S$ of $K$.  This is, by
definition,
\[ H_S^*(X) := H^*((ES \times X)/S), \]
where $ES$ denotes a contractible space with a free $S$-action.  $H_S^*(X)$ is an algebra for the ring
$\Lambda_S := H_S^*(\{\text{pt.}\})$, the $\Lambda_S$-action being given by pullback through the
obvious map $X \rightarrow \{\text{pt.}\}$.

Taking $X$ to be the flag variety $G/B$, we have the following description of $H_S^*(X)$:

\begin{prop}\label{prop:eqvt-cohom-flag-var}
Let $R = S(\frt^*)$, $R' = S(\frs^*)$.  Then $H_S^*(X) = R' \otimes_{R^W} R$.  Thus elements of
$H_S^*(X)$ are represented by polynomials in variables $x_i := 1 \otimes X_i$ and $y_i := Y_i \otimes 1$.
\end{prop}

In the equal rank case, where $S=T$, the statement of this theorem is the standard fact that
$H_T^*(X) = R \otimes_{R^W} R$, for which a proof can be found in \cite{Brion-98_i}.  As we have noted, this is the case in $6$ of our $7$ examples.

Next, we recall the standard localization theorem for torus actions, which can also be found in
\cite{Brion-98_i}:

\begin{theorem}\label{thm:eqvt-localization}
Let $X$ be an $S$-variety, and let $i: X^S \hookrightarrow X$ be the inclusion of the $S$-fixed locus of $X$.  The pullback map of $\Lambda_S$-modules 
\[ i^*: H_S^*(X) \rightarrow H_S^*(X^S) \]
is an isomorphism after a localization which inverts finitely many characters of $S$.  In particular, if $H_S^*(X)$ is free over $\Lambda_S$, then $i^*$ is injective. 
\end{theorem}

When $X$ is the flag variety, $H_S^*(X) = R' \otimes_{R^W} R$
is free over $R'$, so any equivariant class is entirely determined by its image under $i^*$.
Further, the $S$-fixed locus coincides with the $T$-fixed locus.  This is of course obvious when $S=T$.  In the more general case, and in particular in case (7), we have the following result:

\begin{prop}[\cite{Brion-99}]\label{prop:s-fixed-points}
If $K = G^{\theta}$ is a symmetric subgroup of $G$, $T$ is a $\theta$-stable maximal torus of $G$, and
$S$ is a maximal torus of $K$ contained in $T$, then $(G/B)^S = (G/B)^T$.  Thus $(G/B)^S$ is finite, and
indexed by the Weyl group $W$.
\end{prop}

So for us, 
\[ H_S^*(X^S) \cong \bigoplus_{w \in W} \Lambda_S, \]
so that in fact a class in $H_S^*(X)$ is determined by its image under $i_w^*$ for each $w \in W$,
where here $i_w$ denotes the inclusion of the $S$-fixed point $wB/B$.  Given a class $\beta \in H_S^*(X)$
and an $S$-fixed point $wB/B$, we will typically denote the restriction $i_w^*(\beta)$ at $wB/B$ by $\beta|_w$.

We recall how the restriction maps are computed:
\begin{prop}\label{prop:restriction-maps}
Suppose that $\beta \in H_S^*(X)$ is represented by the polynomial $f = f(x,y)$ in the $x_i$ and $y_i$.  Then $\beta|_w \in \Lambda_S$ is the polynomial $f(\rho(wX),Y)$, with $\rho$ denoting restriction $\frt^* \rightarrow \frs^*$.
\end{prop}

By the standard background covered above, we see that for a given closed $K$-orbit $Y$,
the equivariant class $[Y]$ is completely determined by its restrictions $[Y]|_w$ for $w \in W$.  The idea, then, is to compute these restrictions and then try to ``guess" a formula for $[Y]$ (that is, a polynomial in the $\xx$ and $\yy$-variables of Proposition \ref{prop:eqvt-cohom-flag-var} which represents $[Y]$) based on them.

The following proposition tells us precisely how to compute $[Y]|_w$.
\begin{prop}\label{prop:restriction-of-closed-orbit}
Let $\Phi$ denote the root system for $(G,T)$, and choose $\Phi^+$ to be the positive system for $G$ such that the roots of the Borel $B$ are $-\Phi^+$.  Let $\Phi_K$ denote the roots of $K$.  Consider the (multi)-set of weights $\rho(w \Phi^+) \subset \frs^*$, where $\rho: \frt^* \rightarrow \frs^*$ denotes the restriction map.  Let $\mathcal{R} = \rho(w \Phi^+) - \Phi_K$.  For a closed $K$-orbit $Y$, we have 
\begin{equation}\label{eqn:compute-restrictions}
[Y]|_w = 
\begin{cases}
	\displaystyle\prod_{\alpha \in \mathcal{R}} \alpha & \text{ if $w \in Y$}, \\
	0 & \text{ otherwise.}
\end{cases}
\end{equation}
\end{prop}

This is proved in \cite{Wyser-13-TG}, and follows from the self-intersection formula.  As mentioned previously, in cases (1)-(6) when $S=T$, the map $\rho$ is simply given by $X_i \mapsto Y_i$, and we generally omit it from the notation.

\subsection{Closed orbits and their torus fixed points}\label{sec:closed-orbits-and-fixed-points}
It is clear from the general arguments of Section \ref{sec:eqvt-cohomology-background} that to compute the classes of
closed orbits, we need to know which orbits are closed with respect to the
parametrizations we have described, as well as which torus fixed points are contained in each.  This can be easily
determined from known results, as we now describe.  References for this section are
\cite{Matsuki-Oshima-90,Yamamoto-97,Wyser-Thesis}.

Recalling the parametrization of $K$-orbits by clans described in Sections \ref{sec:type-a-case} and
\ref{sec:other-cases}, the closed orbits correspond precisely to those clans consisting of only signs in all cases but
(7).  In case (7), the closed orbits correspond to the symmetric $(2p+1,2q-1)$-clans of the form
$\pm^{(n-1)} 1 1 \pm^{(n-1)}$, i.e. symmetric clans consisting of $n-1$ signs, followed by a pair of matched numbers,
followed by another $n-1$ signs.  (Note that by parity considerations, there are no symmetric $(2p+1,2q-1)$-clans
consisting only of signs.)

Now, we describe which torus fixed points are contained in which closed orbits.  It is easy to see that any closed
$K$-orbit is isomorphic to the flag variety for $K$ if $K$ is connected, or to a disjoint union of
$[W_K: W_{K^0}]$ copies of the flag variety for $K^0$ otherwise.  Thus each closed orbit contains $|W_K|$ $S$-fixed
points.  Moreover, it is clear that the $S$-fixed points contained in a given closed orbit $K \cdot wB/B$ ($wB/B$ an
$S$-fixed point, with $w \in W$) are precisely the elements of the left coset $W_Kw \subseteq W$.

In general, though, the orbit of any particular $S$-fixed point may not be closed.  The only time this actually
happens for us is in case (7).  So we first describe the situation for cases (1)-(6), then consider case (7)
separately.

\subsubsection{Equal rank:  Cases (1)-(6)}\label{sec:equal-rank-cases}
By results of \cite{Springer-85,Richardson-Springer-90}, the orbit of every $S$-fixed point is closed if and only if
$\text{rank}(K) = \text{rank}(G)$, if and only if $K=G^{\theta}$ for an inner involution $\theta$.  These equivalent
conditions all hold in cases (1)-(6).

Thus in each case, there are $|W|/|W_K|$ closed orbits, each containing $|W_K|$ $S$-fixed points.   Note that for
cases (2) and (5), where $K$ has $2$ components, the closed orbits each have $2$ components, each the union of $2$
distinct $K^0$-orbits.  In these cases, if $wB/B$ is an $S$-fixed point, its $K$-orbit contains precisely the
$S$-fixed points corresponding to $W_Kw \subseteq W$, whereas its $K^0$-orbit contains only those $S$-fixed points
corresponding to $W_{K^0}w$.  (Recall that $W_{K^0}$ is of index $2$ in $W_K$ in each case; cf. Section
\ref{sec:specific-cases}.)

With the above observations made, it suffices to simply describe how to give a single $S$-fixed representative of
each closed orbit as a function of the corresponding clan, which consists only of signs.  For case (1), there is an
algorithm given in \cite{Yamamoto-97} which produces a representative of the $K$-orbit indexed by any clan.  When the
clan consists only of signs, this algorithm always outputs an $S$-fixed point corresponding to a permutation whose
one-line notation places $1,\hdots,p$ (in arbitrary order) in the positions of the $+$ signs, and $p+1,\hdots,n$ (in
arbitrary order) in the positions of the $-$ signs.  The particular permutation output by the algorithm depends upon
a choice; it is natural to choose the permutation whose one-line notation places both $1,\hdots,p$ and $p+1,\hdots,n$
in ascending order.  So for instance, the orbit indexed by $(3,3)$-clan $+++---$ contains the standard coordinate
flag, which corresponds to the identity permutation $123456$.  The orbit corresponding to the $(5,2)$-clan $+-+++-+$
contains the $S$-fixed corresponding to $1623475$.

From this description of the closed $K$-orbits in type $A$, together with the fact that the orbits in the other cases
are intersections of these with a smaller flag variety of type $BCD$, it is also easy to specify an $S$-fixed
representative of each closed orbit in cases (2)-(6).  One chooses a representative in one of the following two ways:

\begin{itemize}
	\item Cases (2), (3), (5):  Given a symmetric $(2p,2q+1)$ or $(2p,2q)$-clan $\gamma$, considering only the first $n=p+q$ characters (which are necessarily $p$ $+$ signs and $q$ $-$ signs), choose a signed permutation $w$ of $\{1,\hdots,n\}$ \textit{with no sign changes} just as we do in case (1).
	\item Cases (4), (6):  Given a skew-symmetric $(n,n)$-clan $\gamma$, considering only the first $n$ characters (which may be any collection of $n$ signs), choose the signed permutation $w$ whose one-line notation puts $1,\hdots,n$ in order, with bars over those values in positions corresponding to $-$ signs.
\end{itemize}

As examples, a representative of the closed $S(O(4,\C) \times O(5,\C))$-orbit on $SO(9,\C)/B$ corresponding to
$+-+---+-+$ is $1324$ (interpreted as a signed permutation).  A representative of the closed $GL(4,\C)$-orbit on
$Sp(8,\C)/B$ corresponding to $+-+-+-+-$ is $1 \overline{2} 3 \overline{4}$.

\subsubsection{Case (7)}\label{sec:case-7}
Since this is the only case in which the rank of $K$ is less than that of $G$, it is the only case where there exist
$S$-fixed points whose $K$-orbits are not closed.  So we must first determine which $w \in W$ are such that
$K \cdot wB/B$ is closed.

\begin{prop}\label{prop:type-d-ex-3-closed-orbits}
Let $wB/B$ be an $S$-fixed point, with $w \in W$.  Then $K \cdot wB/B$ is closed if and only if $w(n) = \pm(p+1)$.
\end{prop}
\begin{proof}
This follows from \cite[Proposition 1.4.3]{Richardson-Springer-92}, which states the following:

\begin{prop}
For $w \in W$, the $K$-orbit through $wB/B$ is closed if and only if $wBw^{-1}$ is $\theta$-stable.
\end{prop}

Since we have chosen $B$ to be the negative Borel, the condition that $wBw^{-1}$ be $\theta$-stable is equivalent to
the condition that $w \Phi^-$ is a $\theta$-stable subset of $\Phi$.  One checks easily that the action of $\theta$
on $\Phi$ is defined by $\theta(X_i) = X_i$ for $i \neq p+1$, and $\theta(X_{p+1}) = -X_{p+1}$.  Any positive system
contains, for each $i<j$, exactly one of $X_i + X_j$ and $-X_i - X_j$, and exactly one of $X_i - X_j$ and $-X_i + X_j$.
For $i,j \neq p+1$, all such roots are fixed by $\theta$.  Thus for $\theta$-stability, it suffices to focus on roots
of the form $\pm X_i \pm X_{p+1}$, with $i \neq p+1$.  It is easy to check that a positive system is $\theta$-stable
if and only if it contains either $\{X_i-X_{p+1}, X_i+X_{p+1}\}$ or $\{-X_i + X_{p+1}, -X_i-X_{p+1}\}$ for each
$i \neq p+1$.

This holds if and only if $w(n) = \pm(p+1)$.  Recall that $w\Phi^- = \{-wX_i \pm wX_j \mid i < j\}$.  Suppose that
$w(n) = \pm(p+1)$.  Let $i \neq p+1$ be given, with $k = |w|^{-1}(i)$.  Then $-wX_k \pm wX_n$ is either the set
$\{X_i + X_{p+1}, X_i - X_{p+1}\}$ or $\{-X_i + X_{p+1}, -X_i - X_{p+1}\}$, as required.  Conversely, suppose that
$|w(n)| = j \neq p+1$.  Let $k = |w|^{-1}(p+1)$.  Then $-wX_k \pm wX_n$ is either the set
$\{-X_{p+1} + X_j, -X_{p+1} - X_j\}$ or $\{X_{p+1} + X_j, X_{p+1} - X_j\}$, and thus $w \Phi^-$ is not $\theta$-stable.
This establishes the claim.
\end{proof}

Now, we describe a preferred representative of each closed orbit.  Note that any $u \in W$ such that 
$u(n) = \pm(p+1)$ belongs to the same left $W_K$-coset as a unique element $w \in W$ having the following properties:
\begin{enumerate}
	\item $w$ changes no signs.
	\item $w(n) = p+1$.
	\item $w^{-1}(1) < w^{-1}(2) < \hdots < w^{-1}(p)$.
	\item $w^{-1}(p+2) < w^{-1}(p+3) < \hdots < w^{-1}(n)$.
\end{enumerate}

Indeed, as we have noted, elements of $W_K$ are precisely those separately signed permutations of $\{1,\hdots,p\}$ and
$\{p+2,\hdots,n\}$ which either fix $p+1$ or send it to its negative so as to ensure that the entire
signed permutation changes an even number of signs.  Supposing that in the one-line notation for $u$, the 
values $1,\hdots,p$ (possibly with minus signs) occur out of order, then there is precisely one signed permutation
of $\{1,\hdots,p\}$ which will put them in order and remove all negative signs, and likewise for the set
$\{p+2,\hdots,n\}$.  Taking $w' \in W_K$ to be this element, we have that $w'u = w$.

As an example of the above, suppose that $p=q=3$, and let $u$ be the signed permutation $\overline{3} 1 6 2 5\overline{4}$.
To change the $\overline{3}12$ to $123$, we must multiply on the left by $1 \mapsto 2$, $2 \mapsto 3$,
$3 \mapsto \overline{1}$, and to change the $65$ to $56$ we must multiply on the left by
$5 \mapsto 6, 6 \mapsto 5$.  Thus we multiply $u$ on the left by $w' = 23\overline{1}\overline{4}65$ to get
$w'u = w = 125364$.

\begin{definition}
We refer to the unique element $w$ described above as the \textbf{standard representative} of its $K$-orbit.
\end{definition}

Note that the standard representative of a closed $K$-orbit is completely determined by the positions
(in its one-line notation) of $1,\hdots,p$ among the first $n-1$ spots, which can be chosen freely.
There are $\binom{n-1}{p}$ such $w$, and hence $\binom{n-1}{p}$ closed $K$-orbits.  These correspond to the symmetric
$(2p+1,2q-1)$-clans of the form $\pm^{(n-1)} 11 \pm^{(n-1)}$ mentioned in Section \ref{sec:specific-cases}.
Indeed, given such a clan, one reads off the standard representative of the orbit in a manner very similar to previous
cases:  Considering only the first $n-1$ characters of such a clan (comprised of $p$ $+$ signs and $q$ $-$ signs in
arbitrary order), the one-line notation of the standard representative places $1,\hdots,p$ in order in the positions
of the $+$ signs, $p+2,\hdots,n$ in order in the positions of the $-$ signs, and $p+1$ in position $n$.

We remark that the closed $K$-orbits are connected in this case, even though $K$ is not.  This can be deduced from a
simple counting argument or, alternatively, from the fact that $W_K = W_{K^0}$.

\subsection{Statement and proof of the main result}
In this section, we give formulas for the equivariant classes of closed $K$-orbits in our various examples.  In each case, the method of proof is as follows.  Using the information of Section \ref{sec:closed-orbits-and-fixed-points} on the closed orbits and the torus fixed points contained in each, combined with Proposition \ref{prop:restriction-of-closed-orbit}, we compute the restriction of the class of each closed orbit to each of the torus-fixed points.   Finally, employing Proposition \ref{prop:restriction-maps}, we verify that our putative formulas localize as required, which proves their correctness by Theorem \ref{thm:eqvt-localization}.

We start by summarizing our results in tabular format, so that they are all in one place and easily accessible, as opposed to being sprinkled over several subsections.  So that the formulas in the table will make sense, we first define several notations which appear in them.

\begin{definition}
Given any permutation $w \in S_{p+q}$, denote by $l_p(w)$ the number 
\[ l_p(w) := \# \{ (i,j) \ \vert \ 1 \leq i < j \leq n, w(j) \leq p < w(i) \}. \]
\end{definition}

\begin{definition}
Given any signed permutation $w$ of $\{1,\hdots,n\}$, define the (ordinary) permutation $|w|$ by $|w|(i) = |w(i)|$.  For example, $|\overline{2}\overline{4}13\overline{5}| = 24135$.
\end{definition}

\begin{definition}
Given any signed permutation $w$ of $\{1,\hdots,n\}$, and $p < n$, define 
\[ \phi_p(w) = \#\{i \in \{1,\hdots,n\} \mid w(i) < 0, |w(i)| \leq p\}. \]

For example, if
\[ n = 5, p = 3, w = \overline{2}\overline{4}13\overline{5}, \]
then $\phi_3(w) = 1$.
\end{definition}

We remark that the function $\phi_p$ does not appear in Table \ref{tab:formulas}, but it is in fact referred to and used in the proof of Theorem \ref{thm:formulas}.

\begin{definition}
For any signed permutation $w$ of $n$ elements, define the set 
\[ \text{Neg}(w) := \{i \ \vert \ w(i) < 0 \}. \]
Define $\N$-valued functions $\psi$, $\sigma$ on the set of such permutations by
\[ \psi(w) = \#\text{Neg}(w), \]
and
\[ \sigma(w) = \displaystyle\sum_{i \in \text{Neg}(w)}(n - i). \]
\end{definition}

\begin{definition}
Given a signed permutation $w$ of $n$ elements, denote by $\Delta_n(\xx,\yy,w)$ the $n \times n$ determinant 
\[ \Delta_n(\xx,\yy,w) := \det(c_{n+1+j-2i}), \]
where
\[ c_k = e_k(x_{w^{-1}(1)},\hdots,x_{w^{-1}(n)}) + e_k(y_1,\hdots,y_n). \]
Here, $e_k$ denotes the $k$th elementary symmetric function in the inputs, and $x_{w^{-1}(i)}$ means $x_{w^{-1}(i)}$ if $w^{-1}(i) > 0$, and $-x_{|w^{-1}(i)|}$ if $w^{-1}(i) < 0$.
\end{definition}

\begin{definition}
Given a signed permutation $w$ of $n$ elements, and any $p<n$, denote by 
\[ I_{w,p} := \{ i \in \{1,\hdots,n-1\} \mid w(i) > p+1\}. \]
For each $i \in I(w,p)$, define
\[ C_{w,p}(i) := \{j \mid i < j \leq n-1, w(j) \leq p\}. \]
Finally, define an $\N$-valued function $\tau_p$ on such signed permutations by
\[ \tau_p(w) := \displaystyle\sum_{i \in I_{w,p}} \#C_{w,p}(i). \]
\end{definition}

\begin{theorem}\label{thm:formulas}
The formulas for the equivariant classes of closed $K$-orbits in all of our cases are given in Table \ref{tab:formulas}
(on the next page).  In all cases, $n = p+q$.
\end{theorem}
\newpage
\begin{landscape}
\begin{table}[!ht]
\begin{center}
\resizebox{21cm}{4cm}{
\begin{tabular}{|l|l|l|l|}
\hline
		Case & $(G,K)$ & Parameters for Closed Orbits & Formula for $K \cdot wB/B$ \\ \hline \hline
		(1) & $(GL(n,\C),GL(p,\C) \times GL(q,\C))$ & $(\pm)_{p,q}$ & $(-1)^{l_p(w)} \displaystyle\prod_{i \leq p < j}(x_{w^{-1}(i)} - y_j)$ \\ \hline
		(2) & $(SO(2n+1,\C),S(O(2p,\C) \times O(2q+1,\C)))$ & Symmetric $(\pm)_{2p,2q+1}$ & $(-1)^{l_p(|w|)} x_{|w|^{-1}(1)} \hdots x_{|w|^{-1}(p)} \displaystyle\prod_{i \leq p < j}(x_{w^{-1}(i)} - y_j)		(x_{w^{-1}(i)} + y_j)$ \\ \hline
		(3) & $(Sp(2n,\C),Sp(2p,\C) \times Sp(2q,\C))$ & Symmetric $(\pm)_{2p,2q}$ & $(-1)^{l_p(|w|)} \displaystyle\prod_{i \leq p < j}(x_{w^{-1}(i)} - y_j)(x_{w^{-1}(i)} + y_j)$ \\ \hline
		(4) & $(Sp(2n,\C),GL(n,\C))$ & Skew-symmetric $(\pm)_{n,n}$ & $(-1)^{\psi(w)+\sigma(w)} \Delta_n(\xx,\yy,w)$ \\ \hline
		(5) & $(SO(2n,\C),S(O(2p,\C) \times O(2q,\C)))$ & Symmetric $(\pm)_{2p,2q}$ & 
		$(-1)^{l_p(|w|)} \displaystyle\prod_{i \leq p < j}(x_{w^{-1}(i)} - y_j)(x_{w^{-1}(i)} + y_j)$ \\ \hline
		(6) & $(SO(2n,\C),GL(n,\C))$ & Skew-symmetric $(\pm)_{n,n}$ & $(-1)^{\sigma(w)} (\frac{1}{2})^{n-1} \Delta_{n-1}(\xx,\yy,w)$ \\ \hline
		(7) & $(SO(2n,\C),S(O(2p+1,\C) \times O(2q-1,\C)))$ & Symmetric $(\pm,1,1,\pm)_{2p+1,2q-1}$ & $(-1)^{\tau_p(w)} x_1 \hdots x_{n-1} \displaystyle\prod_{i \leq p < p+1<j}(x_{w^{-1}(i)}-y_j)(x_{w^{-1}(i)}+y_j)$ \\ \hline
\end{tabular} }
\end{center}
\caption{Theorem \ref{thm:formulas}:  Formulas for equivariant classes of closed $K$-orbits}
\label{tab:formulas}
\end{table}
\end{landscape}
\newpage

\begin{proof}
We give proofs of our formulas for some of the cases, omitting others which are very similar.  Before beginning, we
give here Table \ref{tab:roots-of-k}, which describes the roots of $K$ in each of our $7$ cases.  It appears here,
rather than in Section \ref{sec:specific-cases}, simply for the reader's ease of reference.
\begin{table}[!ht]
\begin{center}
\begin{tabular}{|l|l|}
\hline
		Case & $\Phi_K$ \\ \hline \hline
		(1) & $\{\pm(Y_i-Y_j) \mid 1 \leq i < j \leq p\} \cup \{\pm(Y_i-Y_j) \mid p < i < j \leq n\}$ \\ \hline
		(2) & $\{\pm (Y_i \pm Y_j) \mid 1 \leq i < j \leq p\} \cup \{\pm(Y_i \pm Y_j) \mid p < i < j \leq n\} \cup \{\pm Y_i \mid p < i \leq n\}$ \\ \hline
		(3) & $\{\pm 2Y_i \mid i=1,\hdots,n\} \cup \{\pm (Y_i \pm Y_j) \mid 1 \leq i<j \leq p\} \cup \{\pm (Y_i \pm Y_j) \mid p<i<j\leq n\}$ \\ \hline
		(4) & $\{\pm(Y_i - Y_j) \mid 1 \leq i < j \leq n\}$ \\ \hline
		(5) & $\{\pm(Y_i \pm Y_j) \mid 1 \leq i < j \leq p\} \cup \{\pm(Y_i \pm Y_j) \mid p < i < j \leq n\}$ \\ \hline
		(6) & $\{\pm(Y_i - Y_j) \mid 1 \leq i < j \leq n\}$ \\ \hline
		(7) & $\{\pm Y_i \mid i \neq p+1\} \cup \{\pm(Y_i \pm Y_j) \mid 1 \leq i < j \leq p\} \cup \{\pm(Y_i \pm Y_j) \mid p+1 < i < j \leq n \}$ \\ \hline
\end{tabular}
\end{center}
\caption{Root systems for $K$ in the various cases}
\label{tab:roots-of-k}
\end{table}

\begin{proof}[Case (1)]
Let $Q= K \cdot wB/B$ be a closed orbit, corresponding to a $(p,q)$-clan $\gamma$ consisting only of signs.  Recall that $w$ can be taken to be any permutation whose one-line notation places $1,\hdots,p$ in the positions of the $+$ signs of $\gamma$, and $p+1,\hdots,n$ in the remaining positions.

First, observe that the formula given is independent of the choice of $w$.  Indeed, any other $S$-fixed point
in $Q$ is of the form $w'w$ for some $w' \in W_K = S_p \times S_q$.  Since $w'$ preserves the sets $\{1,\hdots,p\}$
and $\{p+1,\hdots,n\}$, we see that 
\[ w'w(j) \leq p < w'w(i) \Leftrightarrow w(j) \leq p < w(i), \]
and so $l_p(w'w) = l_p(w)$.  Further, the set $\{w^{-1}(i) \mid i \leq p \}$ (that is, the set of indices on the
$x$'s in our putative formula) is clearly the same as $\{(w'w)^{-1}(i) \mid i \leq p \} =
\{(w^{-1}w'^{-1})(i) \mid i \leq p\}$, since $w'^{-1}$ permutes those $i$ which are less than or equal to $p$.

We now use Proposition \ref{prop:restriction-of-closed-orbit} to identify the restriction of $[Q]$ at each $S$-fixed point.  The set $\rho(w\Phi^+)$ is
\[ \rho(\{w \alpha \mid \alpha \in \Phi^+ \}) = \{ Y_{w(i)} - Y_{w(j)} \mid i < j \}. \]

Discarding roots of $K$ (cf. Table \ref{tab:roots-of-k}) from this set, we are left with precisely one of $\pm (Y_i - Y_j)$ for each $i,j$ with $i \leq p < j$.  The number of remaining roots of the form $-(Y_i-Y_j)$ is precisely $l_p(w)$.  Thus 
\[ [Q]|_w = F(Y) := (-1)^{l_p(w)} \displaystyle\prod_{i \leq p < j}(Y_i - Y_j). \]

(Note that because $l_p$ is constant on cosets $W_K w$, the restriction $[Q]|_w$ is actually the same at every $S$-fixed point $w \in Q$.)

So for any $u \in W$,
\[ [Q]|_u = 
\begin{cases}
	F(Y) & \text{ if $uB/B \in Q$}, \\
	0 & \text{ otherwise.}
\end{cases} \]
Recalling the precise definition of the restriction maps $i_u^*$ given in Proposition \ref{prop:restriction-maps}, we
see that we are looking for a polynomial $P$ in the $\xx$ and $\yy$ variables such that $P(w'w(Y),Y) = F(Y)$
if $w' \in W_K$, and $0$ otherwise.

So take $P$ to be the representative given in Table \ref{tab:formulas}.  It is straightforward to check that $P$ has
the required properties.  Indeed, for $w' \in W_K$, we see that 
\[ P(w'w(Y),Y) = (-1)^{l_p(w)} \displaystyle\prod_{i \leq p < j}(Y_{w'(i)} - Y_j), \]
and since $w'$ permutes $\{1,\hdots,p\}$, this is precisely $F(Y)$.

On the other hand, given $w' \notin W_K$,
\[ P(w'w(Y),Y) = (-1)^{l_p(w)} \displaystyle\prod_{i \leq p < j}(Y_{w'(i)} - Y_j) = 0, \]
since $w'$, not being an element of $W_K$, necessarily sends some $i \leq p$ to some $j > p$, causing the factor
$Y_{w'(i)} - Y_j$ to be equal to $0$.

We conclude that $P$ represents $[Q]$.
\end{proof}

\begin{proof}[Case (2)]
Note that here $K$ is disconnected.  Each closed $K$-orbit corresponds to a symmetric $(2p,2q+1)$-clan consisting only of $+$'s and $-$'s, and is a union of $2$ closed
$K^0 = SO(2p,\C) \times SO(2q+1,\C)$-orbits.  We shall obtain formulas for the equivariant classes of these individual
components, then add them to get a formula for each closed $K$-orbit.

We claim that for an $S$-fixed point $wB/B$, the class of the closed $K^0$-orbit $K^0 \cdot wB/B$ is represented by the polynomial
$(-1)^{\phi_p(w) + l_p(|w|)} P(\xx,\yy)$, where
\[ P(\xx,\yy) := \frac{1}{2} (x_{w^{-1}(1)} \hdots x_{w^{-1}(p)} + y_1 \hdots y_p) \displaystyle\prod_{i \leq p < j}(x_{w^{-1}(i)} - y_j)(x_{w^{-1}(i)} + y_j). \]

First, let us check that this formula is independent of the choice of $w$.  Recall from Section \ref{sec:specific-cases} the description of $W_{K^0}$ as the subgroup of $W$ consisting of signed permutations of $\{\pm 1,\hdots,\pm n\}$ which act separately on $\{\pm 1,\hdots, \pm p\}$ and $\{ \pm (p+1),\hdots, \pm n \}$, and which change an even number of signs on the first set.  Then the function
$\phi_p$ is constant modulo $2$ on right cosets $W_{K^0}w$, since elements of $W_{K^0}$ permute
$\{1,\hdots,p\}$ with an even number of sign changes.  Considering $l_p(|w|)$, note that if $w' = w_K w$
for $w_K \in W_{K^0}$, then $|w'| = |w_K| |w|$, and $|w_K|$ is an ordinary permutation of $\{1,\hdots,n\}$
which acts separately on $\{1,\hdots,p\}$ and $\{p+1,\hdots,n\}$.  This implies that $l_p(|w'|) = l_p(|w|)$.

Next, consider the term 
\[ x_{w^{-1}(1)} \hdots x_{w^{-1}(p)} + y_1 \hdots y_p. \]

Replacing $w$ by $w_K w$, we get
\[ x_{w^{-1}(w_K^{-1}(1))} \hdots x_{w^{-1}(w_K^{-1}(p))} + y_1 \hdots y_p = x_{w^{-1}(1)} \hdots x_{w^{-1}(p)} + y_1 \hdots y_p, \]
since $w_K$ permutes $\{1,\hdots,p\}$ with an even number of sign changes.  Finally, to see that the product 
\[ \displaystyle\prod_{i \leq p < j}(x_{w^{-1}(i)} - y_j)(x_{w^{-1}(i)} + y_j) \]
also does not depend on the choice of $w$, it is perhaps easiest to note that this expression is unchanged
if we replace $w$ by $|w|$.  This reduces matters to the same type of argument given in case (1).

With this established, we apply Proposition \ref{prop:restriction-of-closed-orbit} to compute the
restrictions $[Q]|_w$.  Recall our choice of $\Phi^+$ described in Section \ref{sec:specific-cases}.
Now, applying $w$ to positive roots of the form $X_i$, we obtain $X_{w(i)} = \pm X_j$
for $j=1,\hdots,n$.  Applying $w$ to $X_i \pm X_j$, we obtain, for each $k < l$, exactly one of
$\pm (X_k + X_l)$, and exactly one of $\pm (X_k - X_l)$.  Restricting to $\frs$ (i.e. replacing $X$'s
with $Y$'s) and eliminating roots of $K$ (cf. Table \ref{tab:roots-of-k}), we are left with $\pm Y_j$ with $j \leq p$, along with, for
each $k \leq p < l$, exactly one of $\pm (Y_k + Y_l)$, and exactly one of $\pm (Y_k - Y_l)$.

The number of weights of the form $\pm Y_j$ occurring with a negative sign is clearly $\phi_p(w)$.  It is
an easy argument to determine that the number of weights of the latter type occurring with a minus sign is
congruent modulo $2$ to $l_p(|w|)$.  So for any $S$-fixed point $w \in Q$,
\[ [Q]|_w = (-1)^{\phi_p(w)+l_p(|w|)} F(Y), \]
where
\[ F(Y) := Y_1 \hdots Y_p \displaystyle\prod_{i \leq p < j}(Y_i + Y_j)(Y_i - Y_j). \]

Thus we must prove that
\[ P(\sigma Y,Y) = 
\begin{cases}
	F(Y) & \text{ if $\sigma = w'w$ for some $w' \in W_K$}, \\
	0 & \text{ if $\sigma w^{-1} \notin W_K$.}
\end{cases}
\]

First, we establish this when $Q$ is the orbit containing the $S$-fixed point corresponding to the identity.
The general case follows easily.  Suppose first that $w \in W_{K^0}$.  Since $w$ permutes $\{1,\hdots,p\}$
with an even number of sign changes, we have
\[ Y_{w(1)} \hdots Y_{w(p)} = Y_1 \hdots Y_p. \]

Further, again because $w$ permutes $\{1,\hdots,p\}$ with an even number of sign changes, we see that
\[ \displaystyle\prod_{i \leq p < j}(Y_{w(i)} + Y_j)(Y_{w(i)} - Y_j) = \displaystyle\prod_{i \leq p < j}(Y_i + Y_j)(Y_i - Y_j). \]

This says that
\[ P(wY,Y) = Y_1 \hdots Y_p \displaystyle\prod_{i \leq p < j}(Y_i + Y_j)(Y_i - Y_j) \]
for $w \in W_{K^0}$.

Now, suppose $w \notin W_{K^0}$.  Then one of two things is true:  Either $w$ is separately a signed
permutation of $\{1,\hdots,p\}$ and $\{p+1,\hdots,n\}$, but permutes $\{1,\hdots,p\}$ with an \textit{odd}
number of sign changes, or $w$ is not separately a signed permutation of $\{1,\hdots,p\}$ and
$\{p+1,\hdots,n\}$, in which case $w$ sends some $i \leq p$ to $\pm j$ for some $j > p$.  In the former
case, we see that 
\[ Y_{w(1)} \hdots Y_{w(p)} + Y_1 \hdots Y_p = 0, \]
while in the latter case, either $Y_{w(i)} + Y_j = 0$, or $Y_{w(i)} - Y_j = 0$, whence
\[ \displaystyle\prod_{i \leq p < j}(Y_{w(i)} + Y_j)(Y_{w(i)} - Y_j) = 0. \]

Together, these two facts say that 
\[ P(wY,Y) = 0 \]
whenever $w \notin W_{K^0}$.  We conclude that $P(x,y)$ represents $[Q]$.

Now, suppose that $\wt{Q}$ is another closed $K^0$-orbit, containing the $S$-fixed point $w \notin W_{K^0}$.
All $S$-fixed points contained in $\wt{Q}$ are then of the form $w'w$ for $w' \in W_{K^0}$.  So for any
$w'w \in \wt{Q}$, we have
\[ P(w'wY,Y) = \frac{1}{2} (Y_{w'(1)} \hdots Y_{w'(p)} + Y_1 \hdots Y_p) \displaystyle\prod_{i \leq p < j}(Y_{w'(i)} - Y_j)(Y_{w'(i)} + Y_j) = \]
\[ Y_1 \hdots Y_p \displaystyle\prod_{i \leq p < j}(Y_i + Y_j)(Y_i - Y_j), \]
by our previous argument, since $w' \in W_{K^0}$.  Noting that this is precisely what $P(w'wY,Y)$ is to be
up to sign, and noting that we have corrected the sign by the appropriate factor of
$(-1)^{\phi_p(w) + l_p(|w|)}$ in our putative formula, we see that it restricts correctly at $S$-fixed
points contained in $\wt{Q}$.

On the other hand, for any $S$-fixed point $\wt{w}$ \textit{not} contained in $\wt{Q}$, we may write
$\widetilde{w} = w'w$ for $w' \notin W_{K^0}$.  Then
\[ P(\wt{w}Y,Y) = P(w'wY,Y) = \]
\[ \frac{1}{2} (Y_{w'(1)} \hdots Y_{w'(p)} + Y_1 \hdots Y_p) \displaystyle\prod_{i \leq p < j}(Y_{w'(i)} - Y_j)(Y_{w'(i)} + Y_j) = 0, \]
again by our previous argument, since $w' \notin W_{K^0}$.

With formulas for the closed $K^0$-orbits now in hand, we recall that each closed $K$-orbit splits up as a union of $2$ distinct closed $K^0$-orbits.  One can pass from one component of a closed $K$-orbit to the other by multiplication by an element $\pi$ of $W_K$ which is not in $W_{K^0}$, i.e. by a signed permutation of $\{\pm 1,\hdots, \pm n\}$ which acts separately on $\{\pm 1,\hdots, \pm p\}$ and $\{\pm (p+1),\hdots,\pm n\}$ and changes an odd number of signs on the first set.  One natural choice of $\pi$ is $\overline{1}2 \hdots n$.

So $Q_{\gamma} = K^0 \cdot wB/B \cup K^0 \cdot \pi wB/B$, which implies that $[Q_{\gamma}] = [K^0 \cdot wB/B] + [K^0 \cdot \pi wB/B]$.
Using the formulas obtained above for the closed $K^0$-orbits, this sum simplifies to the
formula given in Table \ref{tab:formulas} when one makes the following easy observations:
\begin{enumerate}
	\item $\phi_p(\pi w) = \phi_p(w) + 1$;
	\item $l_p(|w|) = l_p(|\pi w|)$;
	\item $x_{w^{-1}(1)} \hdots x_{w^{-1}(p)}$ = $-x_{(\pi w)^{-1}(1)} \hdots x_{(\pi w)^{-1}(p)}$;
	\item $(-1)^{\phi_p(w)} x_{w^{-1}(1)} \hdots x_{w^{-1}(p)}$ = $x_{|w|^{-1}(1)} \hdots x_{|w|^{-1}(p)}$.
\end{enumerate}
\end{proof}

The proofs of the correctness of the formulas in cases (3) and (5) is very similar to that for case (2), so we omit them.

\begin{proof}[Case (4)]
Let us again start by noting that the formula for the class of $Q = K \cdot wB/B$ given in Table \ref{tab:formulas} is independent of the choice of fixed point $w$ representing $Q$.  First, since fixed points contained in $Q$ are all $W_K$ translates of $w$, and since elements of $W_K$ are signed permutations which change no signs, all the fixed points in $Q$ correspond to $W$-elements having the same ``sign pattern".  By this, we mean that the set $\text{Neg}(w)$ is independent of the choice of $w$.  As an example, when $n = 2$, the $4$ closed $K$-orbits on $G/B$ contain $S$-fixed points $\{12,21\}$, $\{\overline{1}2,\overline{2}1\}$, $\{1\overline{2},2\overline{1}\}$, and $\{\overline{1} \overline{2}, \overline{2} \overline{1}\}$.

It follows that the functions $\psi, \sigma$ are constant on $W_K w$, so that $(-1)^{\psi(w) + \sigma(w)}$ is independent of the choice of $w$.  It is also easy to see that $\Delta_n(\xx,\yy,w)$ is independent of this choice.  Indeed, replacing $w$ by $w'w$ for $w' \in W_K$, each $c_k$ becomes
\[ e_k(x_{w^{-1}(w'^{-1}(1))},\hdots,x_{w^{-1}(w'^{-1}(n))}) + e_k(y_1,\hdots,y_n). \]
Because $w'$ is just an ordinary permutation of $\{1,\hdots,n\}$, the effect is simply to permute the $x_{w^{-1}(i)}$, and because $e_k$ is invariant under permutation of the inputs, each $c_k$ is unchanged.

With this established, let us apply Proposition \ref{prop:restriction-of-closed-orbit} to compute the restriction
$[Q]|_w$.  Applying $w$ to positive roots of the form $2X_i$, we get weights of the form
$2Y_{w(i)} = \pm 2Y_j$.  The number of such weights occurring with a minus sign is $\psi(w)$.  Note that none of
these weights are roots of $K$ (cf. Table \ref{tab:roots-of-k}).

Applying $w$ to positive roots of the form $X_i \pm X_j$ ($i<j$), we get weights of the form
$Y_{w(i)} \pm Y_{w(j)}$, and these two weights together are of the form $\pm Y_k \pm Y_l$, $\pm Y_k \mp Y_l$, for
some $k,l$.  Those of the latter form $\pm Y_k \mp Y_l$ are roots of $K$ (cf. Table \ref{tab:roots-of-k}), while those of the former are not.  We
eliminate the latter weights, and retain the former.  The number of roots surviving which are negative (i.e. of the
form $-Y_k - Y_l$) is precisely $\sigma(w)$.  To see this, note that if $w(i)$ is positive, then applying $w$ to any
pair of roots $X_i+X_j,X_i-X_j$ with $i<j$ is going to necessarily give a positive root of the
form $Y_k+Y_l$, where $k=w(i)$, $l = |w(j)|$.  If $w(i)$ is negative, then applying $w$ to any such pair will
necessarily give a negative root of the form $-Y_k-Y_l$.  For any fixed $i$, the number of pairs $\{X_i \pm X_j\}$
with $i<j$ is precisely $n-i$.  So for each $i$ with $w(i)$ negative, $n-i$ negative roots occur, for a total of
$\sigma(w)$ negative roots.

All of this leads to the conclusion that, for any $S$-fixed point $w \in Q$, we have
\[ [Q]|_w = F(Y) := (-1)^{\psi(w)+\sigma(w)} 2^n Y_1 \hdots Y_n \displaystyle\prod_{i < j}(Y_i + Y_j). \]

So letting $P$ be the polynomial given in Table \ref{tab:formulas}, the claim is thus that $P(uY,Y)$ is $F(Y)$ if
$uw^{-1} \in W_K$, and is $0$ otherwise.  If $uw^{-1} \in W_K$, then it is an ordinary permutation, with
no sign changes, whereas if $uw^{-1} \notin W_K$, then it is a signed permutation with at least one sign change.
Noting that $P(uY,Y)$ is simply $\Delta_n(uw^{-1}Y,Y,1)$, the claim that $P(\xx,\yy)$ represents $[Q]$ thus amounts
to the claim that $\Delta_n(\xx,\yy,id)$ has the following two properties:
\begin{enumerate}
	\item It is invariant under permutations of the $x_i$ and $y_i$.
	\item If $\epsilon_i = \pm 1$, then
	\[ \Delta_n((\epsilon_1Y_1,\hdots,\epsilon_nY_n),(Y_1,\hdots,Y_n),id) \]
	is zero unless all $\epsilon_i$ are equal to $1$, in which case it is equal to 
	\[ 2^n Y_1 \hdots Y_n \displaystyle\prod_{i < j}(Y_i + Y_j). \]
\end{enumerate}

That $\Delta_n(\xx,\yy,id)$ has these properties is proved directly in \cite[\S 3]{Fulton-96_1}.
\end{proof}

The proof of the correctness of the formulas in case (6) is very similar to that for case (4).  It again relies on
properties of the determinants in question which are established in \cite{Fulton-96_1}.  Because the argument is so
similar, we omit the details.

\begin{proof}[Case (7)]
We start by noting that in this case, the formula for the class of $K \cdot wB/B$ given in Table \ref{tab:formulas}
is \textit{not} independent of the choice of $w$ contained in the orbit.  Indeed, the formula given there
assumes $w$ to be the standard representative of the orbit (cf. Section \ref{sec:case-7} for the definition).  In what follows, we always assume $w$ to be the standard representative.

Recalling our labelling conventions for the $X_i$ and $Y_i$ coordinate functions, and the corresponding definition of $\rho$ in this case, (cf. Section \ref{sec:specific-cases}), we once again apply Proposition \ref{prop:restriction-of-closed-orbit}. 

First, consider $\rho(w \Phi^+)$, the elements of $\frs^*$ obtained by first applying the standard representative $w$ to the positive roots, then restricting to $\frs$.  They are as follows:
\begin{itemize}
	\item $Y_i$ ($i \neq p+1$), with multiplicity $2$.  (One is the restriction of $w(X_i + X_n) = X_{w(i)} + X_{p+1}$, the other the restriction of $w(X_i-X_n) = X_{w(i)} - X_{p+1}$.)
	\item $Y_i + Y_j$ ($i<j$, $i,j\neq p+1$), with multiplicity $1$.
	\item For each $i<j$ with $i,j \neq p+1$, exactly one of $\pm(Y_i-Y_j)$, with multiplicity $1$.
\end{itemize}

Removing roots of $K$ (cf. Table \ref{tab:roots-of-k}), we are left with the following weights:
\begin{itemize}
	\item $Y_i$ ($i \neq p+1$), with multiplicity $1$.
	\item $Y_i + Y_j$ ($i \leq p < p+1 < j$), with multiplicity $1$.
	\item For each $i<j$ with $i \leq p < p+1 < j$, exactly one of $\pm(Y_i-Y_j)$, with multiplicity $1$.
\end{itemize}

Recall that $w$ is an honest permutation, with no sign changes.  This means that the only way to get a weight of the form $-(Y_i-Y_j)$ by the action of $w$ is to apply $w$ to some $X_k - X_l$ ($k<l$) with $w(k) > w(l)$, then restrict.  (Clearly, we want $k,l \neq n$.)  For this root to remain after discarding roots of $K$, it must be the case that $w(k) > p+1$, while $w(l) \leq p$.  Thus for each $k < n$ such that $w(k) > p+1$ (this says that $k \in I_{w,p}$), we count the number of $l$ with $k < l < n-1$ such that $w(l) \leq p$ (this says that $l \in C_{w,p}(k)$).  Adding up the total number of such pairs as we let $k$ range over $I_{w,p}$, we arrive at $\tau_p(w)$.  This says that the number of weights of the form $-(Y_i-Y_j)$ contained in $\rho(w \Phi^+) \setminus (\rho(w \Phi^+) \cap \Phi_K)$ is $\tau_p(w)$.
 
Now we consider the set $\rho(w'w \Phi^+) \setminus (\rho(w'w \Phi^+) \cap \Phi_K)$ with $w' \in W_K$, and compute
the restriction $[Q]|_{w'w}$ at an arbitrary $S$-fixed point.  Since the action of $w'$ on $\frt$ commutes with
restriction to $\frs$, and since $w'$ acts on the roots of $K$ (and hence also on $\rho(\Phi) \setminus \Phi_K$), we
can simply apply $w'$ to the set of weights described in the previous paragraph.  We temporarily forget that some of
those roots are of the form $-(Y_i-Y_j)$ ($i<j$), and add the sign of $(-1)^{\tau_p(w)}$ back in at the end.  So
consider the action of $w' \in W_K$ on the following set of weights, each with multiplicity $1$:
\begin{itemize}
	\item $Y_i$ ($i \neq p+1$)
	\item $Y_i \pm Y_j$ ($i \leq p < p+1 < j$)
\end{itemize}

Since $w'$ acts separately as signed permutations on $\{1,\hdots,p\}$ and $\{p+2,\hdots,n\}$, it clearly sends the set of weights $Y_i \pm Y_j$ to itself, except possibly with some sign changes.  We observe that the number of sign changes must be even.  Suppose first that $w'(Y_i + Y_j)$ is a negative root.  Then it is either of the form $-Y_k-Y_l$ or $-Y_k + Y_l$, with $k=|w(i)|$ and $l=|w(j)|$.  In the former case, $w'(Y_i - Y_j) = -Y_k + Y_l$, also a negative root.  In the latter, $w'(Y_i - Y_j) = -Y_k-Y_l$, again a negative root.  Likewise, if $w'(Y_i - Y_j)$ is a negative root of the form $-Y_k-Y_l$ or $-Y_k+Y_l$, then $w'(Y_i+Y_j)$ is also a negative root, equal to $-Y_k+Y_l$ in the former case, and $-Y_k-Y_l$ in the latter.  Thus the negative roots arising from the action of $w'$ on roots of the form $Y_i \pm Y_j$ occur in pairs.

Now consider roots of the form $Y_i$, $i \neq p+1$.  The action of $w'$ again preserves this set of roots, except possibly with some sign changes.  The number of sign changes could be either even or odd.  (Recall that $w'$ acts with any number of sign changes on $\{1,\hdots,p\}$ and $\{p+2,\hdots,n\}$, and sends $p+1$ either to itself or to $-(p+1)$, whichever ensures that the total number of sign changes for $w'$ is even.)

This discussion all adds up to the following.  The product of the weights $\rho(w'w \Phi^+) \setminus (\rho(w'w \Phi^+) \cap \Phi_K)$ is
\[ [Q]|_{w'w} = (-1)^{\tau_p(w) + \#(\text{Neg}(w') \setminus \{p+1\})} \displaystyle\prod_{i \neq p+1} Y_i \displaystyle\prod_{i \leq p < p+1 < j} (Y_i+Y_j)(Y_i-Y_j). \]

Thus we wish to prove that the polynomial $P$ given in Table \ref{tab:formulas} has the properties that $P(\rho(w'wX),Y)$ is equal to this restriction for all $w' \in W_K$, and that $P(\rho(w'wX),Y) = 0$ whenever $w' \notin W_K$.

Consider first the action of $w'w$ on $P$ for $w' \in W_K$.  Since $w$ sends the set $\{1,\hdots,n-1\}$ to the set $\{1,\hdots,p,p+2,\hdots,n\}$ with no sign changes, the action of $w'w$ on $X_1 \hdots X_{n-1}$ is clearly to send it to $(-1)^{\#(\text{Neg}(w') \setminus \{p+1\})} \displaystyle\prod_{i \neq p+1} X_i$.  Thus applying $w'w$ to $(-1)^{\tau_p(w)} x_1 \hdots x_{n-1}$, then restricting, gives us the portion
\[ (-1)^{\tau_p(w) + \#(\text{Neg}(w') \setminus \{p+1\})} \displaystyle\prod_{i \neq p+1} Y_i \]
of the required restriction.  Now consider the action of $w'w$ on the product 
\[ \displaystyle\prod_{i \leq p < p+1<j}(x_{w^{-1}(i)}+y_j)(x_{w^{-1}(i)} - y_j). \]
We get 
\[ \displaystyle\prod_{i \leq p < p+1<j}(Y_{w'^{-1}(i)}+Y_j)(Y_{w'^{-1}(i)} - Y_j). \]
Since $w'$ acts as a signed permutation on $\{1,\hdots,p\}$, this is clearly the same as
\[ \displaystyle\prod_{i \leq p < p+1<j}(Y_i+Y_j)(Y_i - Y_j), \]
giving us the remaining part of the required restriction.

Now, consider the action of $w'w$ on $P$ for $w' \notin W_K$.  Suppose first that $w'(p+1) \neq \pm(p+1)$.  Then
$w'(i) = \pm(p+1)$ for some $i \neq p+1$.  Let $j = w^{-1}(i)$.  Then the action of $w'w$ sends $x_j$ to $\pm X_{p+1}$,
which then restricts to zero.  Now suppose that $w'(p+1) = \pm(p+1)$.  Then since $w' \notin W_K$, $w'$ must send
some $i \leq p$ to $\pm j$ for some $j > p+1$.  If it sends $i$ to $j$, then $w'w$ applied to the term
$x_{w^{-1}(i)}-y_j$ is zero.  If it sends $i$ to $-j$, then $w'w$ applied to the term $x_{w^{-1}(i)}+y_j$ is zero.
This shows that $P(\rho(w'wX),Y) = 0$ for $w' \notin W_K$, completing the proof.
\end{proof}

This concludes the proof of Theorem \ref{thm:formulas}.
\end{proof}

\subsection{Examples}
For each of our cases, an example calculation is given in the Appendix, in the form of the weak order graph together with a table of formulas.  The formulas for closed orbits are those of Theorem \ref{thm:formulas}, while the others were obtained from these using divided difference operators according to the weak order. 

We remark that there are choices involved in making these calculations, since for any given orbit closure we may have a choice of multiple closed orbits from which to start the recursion.  Additionally, there are in general multiple paths connecting a given closed orbit to any given orbit closure, many of which may correspond to different divided difference operators.  Thus there is a question of well-definedness of these polynomials.  We make no general claim here that the representatives of the classes of closed orbits given in Theorem \ref{thm:formulas} gives rise to a well-defined family of polynomial representatives for the classes of all orbit closures.  However, at least for the examples of the Appendix, we have verified that all possible choices do in fact lead to the same polynomial representatives.

\subsection{Cases approachable by other means}
While we have used equivariant localization to determine representatives for the classes of closed orbits in all cases, we remark that in some of our cases, representatives can be determined by other means.

Indeed, in three of our cases, namely (1), (4), and (6), it is known that a number of the orbit closures, including all of the closed orbits, are Richardson varieties, i.e. intersections of Schubert varieties with opposite Schubert varieties.  The combinatorial translation between parameters for the orbit closures and for the corresponding Richardson varieties is spelled out explicitly in \cite{Wyser-11b,Wyser-12b}.  Using this information, one can obtain representatives for the equivariant classes of the closed orbits by simply multiplying representatives for the equivariant classes of the two Schubert varieties.

In type $A$, the \textit{double Schubert polynomials} are widely accepted as the preferred  representatives of
equivariant Schubert classes.  In case (1), the particular system of representatives for classes of $K$-orbit closures
that one obtains by taking the appropriate products of double Schubert polynomials to represent the classes of closed
orbits is studied in \cite{Wyser-Yong-13}.

In the other classical types, no one particular family of equivariant representatives is universally agreed upon as the preferred one, but various representatives are known.  The reader may consult \cite{Fulton-96_1,Fulton-96_2,Pragacz,Pragacz-Ratajski,Kresch-Tamvakis,Ikeda-Mihalcea-Naruse} and references therein.

Although the representatives obtained by taking products of Schubert classes are quite geometrically natural, verifying correctness via localization is more straightforward for our representatives.  Our representatives are also algebraically nice and easy to describe explicitly, whereas double Schubert polynomials are defined recursively, using divided difference operators.  On the other hand, we aren't aware of any sense in which our representatives are geometrically natural.

We also mention that a formula of M. Brion describes how to write the ordinary cohomology class of an orbit closure as a weighted sum of Schubert classes.  The formula is in terms of a sum over weighted paths in the weak order graph, with each path weighted by a power of $2$ according to how many dashed edges it contains.  Using this formula (along with our knowledge of the weak order in the various cases), one can in principle give a polynomial representative for the ordinary class of any orbit closure, by simply replacing the Schubert classes by Schubert polynomials.  In \cite{Wyser-Yong-13}, it is shown that these are precisely the representatives that one gets if one starts with products of double Schubert polynomials, then specializes from equivariant to ordinary cohomology.

Brion's formula applies only in the non-equivariant case.  If one specializes the equivariant formulas of Theorem \ref{thm:formulas} to ordinary cohomology (by setting all $\yy$-variables to $0$), one obtains representatives of the ordinary classes of $K$-orbit closures.  They are typically different from the sums of Schubert polynomials one obtains using Brion's formula.

\section{Connection to degeneracy loci}\label{sec:deg-loci}
In this section, we describe one application of the formulas obtained in the previous section, realizing the
$K$-orbit closures as universal degeneracy loci of a certain type determined by $K$.  We describe a translation
between our formulas for equivariant fundamental classes of $K$-orbit closures and Chern class formulas for the
fundamental classes of such degeneracy loci.

\subsection{Overview}\label{sec:deg-loci-general}
Let $E:=EG$ be a contractible space with a free action of $G$, and
hence, by restriction, a free action of any subgroup of $G$ (e.g. $T$, $S$, $B$, $K$).  Let $BG = E/G$ be a
classifying space for $G$, and similarly define $BK = E/K$, $BB = E/B$, etc.

Recall that $H_K^*(G/B)$ is naturally a subring of $H_S^*(G/B)$, the subring of $W_K$-invariants \cite{Brion-98_i}.
The $S$-equivariant classes that we have computed using localization and divided difference operators
are in fact elements of this subring.  Now the $K$-equivariant cohomology $H_K^*(G/B)$ is, by definition,
$H^*(E \times^K G/B)$, while the space $E \times^K G/B$ is naturally isomorphic to the space $BK \times_{BG} BB$.  

Given a smooth complex variety $X$ and a complex rank $n$ vector bundle $V \rightarrow X$ with certain presumed
additional structures, we get a map 
\[ X \stackrel{\phi}{\longrightarrow} BK \times_{BG} BB. \]
These additional structures amount to two separate lifts of the classifying map $X \rightarrow BG$ for
$V$, one to $BB$, and the other to $BK$.  In type $A$, the additional structure corresponding to a lift of the 
classifying map to $BB$ is well-known to be a complete flag of subbundles of $V$.  The structure corresponding to a
lift to $BK$ depends, of course, on the particular $K$ we are dealing with.  For example, in the cases of
\cite{Wyser-13-TG}, where $K$ was $SO(n,\C)$ or $Sp(2n,\C)$, the additional structure was a symmetric (resp.
skew-symmetric) non-degenerate bilinear form taking values in the trivial bundle.

Here, in case (1), the additional structure required is a splitting of $V$ as a direct sum of two subbundles of ranks
$p$ and $q$.  As described in Section \ref{sec:deg-loci-other-cases}, given the close relationship of cases (2)-(7)
to case (1), in those cases the appropriate structure is again a splitting of $V$, but with additional required
properties.

Given such a setup, and a clan $\gamma$, we may consider the subvariety $D_{\gamma} \subseteq X$ which is the preimage
of the $K$-orbit closure $Y_{\gamma}$ under $\phi$.  (More precisely, $D_{\gamma}$ is the preimage under $\phi$ of the
isomorphic image of $E \times^K Y_{\gamma} \subseteq E \times^K G/B$ in the space $BB \times_{BG} BK$.)

Describing $D_{\gamma}$ explicitly requires an explicit linear algebraic description of the points of $Y_{\gamma}$.
For case (1), we give this as Theorem \ref{thm:orbit-closures}.  Essentially, $Y_{\gamma}$ is described
as the set of all flags which are in a certain position (prescribed in a precise way by $\gamma$) relative to the
standard splitting of $\C^n$ as $\left\langle e_1,\hdots,e_p \right\rangle \oplus 
\left\langle e_{p+1},\hdots,e_n \right\rangle$.  The
corresponding degeneracy locus $D_{\gamma}$ is then the set of all points $x \in X$ over which the fiber of the flag
on $V$ over $x$ is in this same position relative to the fiber of the splitting over $x$, with the ``position'' of a
flag relative to a splitting being defined in precisely the same way as in the $K$-orbit setting.

The various bundles on $X$ can be realized as pullbacks by $\phi$ of tautological bundles on the universal space, in such
a way that their Chern classes are pullbacks of $S$-equivariant classes represented by our $\xx$ and $\yy$
variables (or perhaps by polynomials in these classes).  Making this translation explicit, and having described the
points of $D_{\gamma}$ explicitly, if we assume that
\begin{equation}\label{eqn:pullback}
	[D_{\gamma}] = [\phi^{-1}(Y_{\gamma})] = \phi^*([Y_{\gamma}]),
\end{equation}
then our equivariant formula for $[Y_{\gamma}]$ gives us, in the end, a formula for $[D_{\gamma}] \in H^*(X)$ in
terms of the Chern classes of the bundles involved.  Equation \eqref{eqn:pullback} holds, for instance, if $\phi$ is
a smooth morphism.  Alternatively, \cite[\S B.3, Lemma 5]{Fulton-YoungTableaux} gives a more general sufficient
condition to guarantee \eqref{eqn:pullback}, namely that $D_{\gamma}$ has the expected codimension, and that there
is an open neighborhood $U$ of a smooth point of $Y_{\gamma}$, defined by equations $h_1,\hdots,h_s$, such that
$\phi^{-1}(U)$ is defined inside $D_{\gamma}$ by equations $h_1 \circ \phi,\hdots,h_s \circ \phi$.

\begin{remark}
In case (1), the only one which we actually make fully explicit, we can avoid assuming that $k=\C$, and that $X$ is
smooth, and work in the Chow groups $A_*(X)$ instead.  In this setting, $D_{\gamma}$ is given an implicit subscheme structure
by virtue of being the preimage of $Y_{\gamma}$ under $\phi$.  Taking this point of view, we need only assume that $k$
is algebraically closed of characteristic not equal to $2$, that $X$ is Cohen-Macaulay, and that $D_{\gamma}$ is of the
expected codimension.  Indeed, the argument is identical to that given in the proof of
\cite[Theorem 8.2, (d)]{Fulton-92}.  This requires knowing that all $K$-orbit closures for this case are Cohen-Macaulay,
which follows from the aforementioned fact that the weak order graph in this case contains only solid edges.
A result of Brion \cite{Brion-03} implies that in such cases, all $K$-orbit closures are Cohen-Macaulay.

In principle, this line of argument applies also to cases (3) and (6).  However, in these cases, we do not
give an explicit description of the corresponding degeneracy loci, since we do not know a set-theoretic
description of the $K$-orbit closures in these cases.

For the remaining cases, $K$-orbit closures need not be Cohen-Macaulay.  In such cases, one can still work in the Chow
groups, and \eqref{eqn:pullback} holds (essentially by definition) provided that $\phi$ is flat of some fixed
relative dimension.

Whichever cohomology theory one prefers, equation \eqref{eqn:pullback} holds in the
``generic'' situation, and should be thought of morally as an insistence that the given additional structures on $V$
(the flag and the splitting) are in suitably general position with respect to one another.
\qed
\end{remark}

\subsection{Set-theoretic descriptions of orbit closures}\label{sec:explicit-orbit-closures}
\subsubsection{Case (1)}
As indicated in the previous section, to explicitly describe the types of degeneracy loci for which the $K$-orbits are ``universal" requires an explicit set-theoretic description of the $K$-orbit closures.  In this section, we provide such a description for the type $A$ pair of case (1).

For any $(p,q)$-clan $\gamma=c_1 \hdots c_n$, and for any $i,j$ with $1 \leq i<j \leq n$, define the following quantities:
\begin{enumerate}
	\item $\gamma(i; +) = $ the total number of plus signs and pairs of equal natural numbers occurring among $c_1 \hdots c_i$;
	\item $\gamma(i; -) = $ the total number of minus signs and pairs of equal natural numbers occurring among $c_1 \hdots c_i$; and
	\item $\gamma(i; j) = $ the number of pairs of equal natural numbers $c_s = c_t \in \N$ with $s \leq i < j < t$.
\end{enumerate}

We first recall the following explicit description of the $K$-orbits themselves, due to Yamamoto \cite{Yamamoto-97}.
Let $\gamma$ be a $(p,q)$-clan, with $Q_{\gamma}$ the corresponding $K$-orbit.  Let $E_p$ denote the linear span
$\left\langle e_1,\hdots,e_p \right\rangle$ of the first $p$ standard
basis vectors.  Let $\widetilde{E_q}$ denote the linear span $\left\langle e_{p+1},\hdots,e_n \right\rangle$ of the
\textit{last $q$} standard basis vectors.  Let $\pi$ denote the projection from $\C^n$ onto the subspace $E_p$.
\begin{theorem}\label{thm:orbits}
With notation as above, $Q_{\gamma}$ is precisely the set of flags $F_{\bullet}$ satisfying the following conditions
for all $i<j$: 
\begin{enumerate}
	\item $\dim(F_i \cap E_p) = \gamma(i;+)$
	\item $\dim(F_i \cap \widetilde{E_q}) = \gamma(i;-)$
	\item $\dim(\pi(F_i) + F_j) = j + \gamma(i;j)$
\end{enumerate}
\end{theorem}

The following theorem describes the (strong) Bruhat order on $K$-orbits explicitly, allowing us to give a similarly explicit description of $Y_{\gamma}$, the closure of $Q_{\gamma}$.  Loosely, the theorem says that, as in the case of type $A$ Schubert varieties, we
pass from the description of an orbit to that of its closure by changing equalities to inequalities.
\begin{theorem}\label{thm:orbit-closures}
Given two clans $\gamma$ and $\tau$, $Y_{\tau} \subseteq Y_{\gamma}$ if and only if
\begin{enumerate}
	\item $\tau(i;+) \geq \gamma(i;+)$ for all $i$;
	\item $\tau(i;-) \geq \gamma(i;-)$ for all $i$; and
	\item $\tau(i;j) \leq \gamma(i;j)$ for all $i<j$.
\end{enumerate}

From this, it follows that $Y_{\gamma}$ consists precisely of those flags $F_{\bullet}$ satisfying the following
conditions for all $i<j$:
\begin{enumerate}
	\item $\dim(F_i \cap E_p) \geq \gamma(i;+)$
	\item $\dim(F_i \cap \widetilde{E_q}) \geq \gamma(i;-)$
	\item $\dim(\pi(F_i) + F_j) \leq j + \gamma(i;j)$
\end{enumerate}
\end{theorem}

The proof of Theorem \ref{thm:orbit-closures} is a bit combinatorially involved, as it requires 
extensive case-by-case analysis in order to describe the covering relations in the putative Bruhat
order.  It will appear in a separate paper by the author, currently in preparation.

\begin{remark}
We remark that a special case of Theorem \ref{thm:orbit-closures} appears in the Ph.D. thesis of E.Y. Smirnov, cf. \cite[Theorem 3.10]{Smirnov-Thesis}.  Smirnov considers $B$-orbits on the product $Gr(p,n) \times Gr(q,n)$ of Grassmannians, which contains $G/K$ as a dense open subset.  Thus his results apply in particular to the $B$-orbits on $G/K$, which are in a natural bijection with $K$-orbits on $G/B$.  This bijection is known to preserve both the weak and strong Bruhat orders, so any results regarding the Bruhat order on one set of orbits naturally carries over to the other.

Smirnov uses a different parametrization of orbits from ours, so there is some combinatorial translation involved.  His parametrization of $B$-orbits on $G/K$ is by pairs of Young diagrams, each marked with dots in a certain way so as to encode a particular involution in $S_n$.  (This parametrization is described in detail in \cite{Smirnov-Thesis}, as well as in the published article \cite{Smirnov-08}.)

Smirnov's result describes the Bruhat order on $B$-orbits contained in a given $B \times B$-orbit on $Gr(p,n) \times Gr(q,n)$.
In terms of Smirnov's parameters, such $B$-orbits are all those having precisely the same pair of Young diagrams.
Converting from Smirnov's parameters to ours gives the result in our setting for all clans $\tau,\gamma$ such that
$\tau(i;\pm) = \gamma(i;\pm)$ for all $i$.  In Smirnov's setting, the result says that if both Young diagrams
coincide, then Bruhat comparability of the orbit closures is determined strictly by Bruhat comparability of the involutions
encoded by the dots.  In ours, it says that if $\tau(i;\pm) = \gamma(i;\pm)$ for all $i$ (i.e. if $\tau$ and $\gamma$ have $+$'s, $-$'s, first occurrences, and second
occurrences in precisely the same positions), then Bruhat comparability of $Y_{\tau}$ and $Y_{\gamma}$ is determined strictly by Bruhat comparability
of the underlying involutions of $\tau$ and $\gamma$.  It is easy to see that this is precisely the statement of
Theorem \ref{thm:orbit-closures} in this special case.  Indeed, translated to Smirnov's parameters, Theorem \ref{thm:orbit-closures}
says more generally that two $B$-orbits on $G/K$ are Bruhat comparable if and only if the Young diagrams of one
contain those of the other, \textit{and} the involutions encoded by the dots are Bruhat comparable. 
\qed
\end{remark}

\begin{remark}
Besides allowing us to explicate the connection with degeneracy loci here, Theorem \ref{thm:orbit-closures} may be of
independent interest.  For example, it allows one to write down explicit equations that cut out open affine subsets
of $K$-orbit closures.  Studying such affine neighborhoods should allow for insight into the singularities of the 
orbit closures.  This is of interest in its own right, and is also important in representation theory.
See \cite{Wyser-Yong-13,Woo-Wyser-14} for more details and for some partial results along these lines.
\qed
\end{remark}

\subsubsection{Other cases}
To explicitly connect the $K$-orbit closures in cases (2)-(6) to degeneracy loci, one needs a set-theoretic description of the orbit closures along the lines of Theorem \ref{thm:orbit-closures} for each case.  Since every $K$-orbit in these cases is the intersection of a $K'$-orbit on $X'$ with $X$ (with $X'$ a type $A$ flag variety and $K'$ some $S(GL(p,\C) \times GL(q,\C))$), it is clear that the set-theoretic description of any $K$-orbit is the same as that given in the statement of Theorem \ref{thm:orbits}, except that we restrict attention to isotropic or Lagrangian flags meeting the appropriate linear algebraic conditions.  (In type $D$, we should furthermore restrict our attention to isotropic flags in the appropriate ``family", i.e. those lying in the appropriate $SO(2n,\C)$-orbit on the variety of all isotropic flags.)    The obvious hope, then, is that each $K$-orbit \textit{closure} is the intersection of the corresponding $K'$-orbit \textit{closure} with $X$, so that a set-theoretic description of a $K$-orbit closure would be given by ``changing equalities to inequalities" as in Theorem \ref{thm:orbit-closures}, while again restricting our attention to the appropriate subset of isotropic or Lagrangian flags.

It is clear that this obvious guess at least contains the true $K$-orbit closure, but this containment need not be an equality.  Combinatorially, the issue can be framed as follows:  The set $K \backslash X$ is a subset of $K' \backslash X'$.  There are two possible partial orders one can put on the set $K \backslash X$.  On one hand, there is the Bruhat order, corresponding to containment of $K$-orbit closures on $X$.  On the other hand, there is the order on $K \backslash X$ induced by the Bruhat order on $K' \backslash X'$.  These two partial orders may \textit{a priori} be different.

As explained in \cite{Richardson-Springer-90}, given an explicit understanding of the weak order on $K \backslash X$, one can compute the full Bruhat order explicitly using a simple recursive algorithm.  Since we understand the weak order explicitly in all of our examples, one can use a computer to calculate the Bruhat order on $K \backslash X$ and compare it to the order induced by the Bruhat order on $K' \backslash X'$, then test whether these partial orders do in fact coincide.

The results of experiments of this type support the following conjecture.
\begin{conjecture}\label{conj:bruhat-order-other-types}
In cases (2)-(4), the Bruhat order on $K$-orbits coincides with the induced Bruhat order on the appropriate set of clans.  Thus for any $K$-orbit $Q = X \cap Q'$ for $Q'$ the corresponding $K'$-orbit on $X'$, we have that $\overline{Q} = X \cap \overline{Q'}$.  In particular, the description of $\overline{Q}$ as a set of flags is given by Theorem \ref{thm:orbit-closures}, and we simply restrict our attention to the set of flags meeting this description which lie in $X$ --- namely, isotropic flags in the type $B$ case, or Lagrangian flags in the type $C$ cases.
\end{conjecture}

Conjecture \ref{conj:bruhat-order-other-types} has been verified for each of cases (2)-(4) through rank $7$.

Similar experimentation establishes that the analogous conjecture does \textit{not} hold in any of cases (5)-(7).

\begin{fact}\label{fact:type-d-conj-false}
In each of the type $D$ cases (5)-(7), the Bruhat order on $K \backslash X$ is strictly weaker than the order induced by the Bruhat order on $K' \backslash X'$.  Thus for a general $K$-orbit $Q = X \cap Q'$, $\overline{Q}$ is contained in, but is not equal to, $X \cap \overline{Q'}$.
\end{fact}

We give the following examples:
\begin{itemize}
	\item In the case $(G,K) = (SO(8,\C),GL(4,\C))$, the clans $1+-12+-2$ and $12341234$ are related in the Bruhat order on $K' \backslash X'$ (where $K' = GL(4,\C) \times GL(4,\C)$), but are not related in the Bruhat order on $K \backslash X$.
	\item In the case $(G,K) = (SO(8,\C),S(O(4,\C) \times O(4,\C)))$, the clans $+-1122-+$ and $+-1212-+$ are related in the Bruhat order on $K' \backslash X'$ (with $K' = GL(4,\C) \times GL(4,\C)$), but are not related in the Bruhat order on $K \backslash X$.
	\item In the case $(G,K) = (SO(8,\C),S(O(5,\C) \times O(3,\C)))$, the clans $+121323+$ and $+123123+$ are related in the Bruhat order on $K' \backslash X'$ (with $K' = GL(5,\C) \times GL(3,\C)$), but are not related in the Bruhat order on $K \backslash X$.
\end{itemize}

\subsection{$K$-orbit closures as degeneracy loci:  Case (1)}\label{sec:deg-loci-case-1}
We now describe the degeneracy locus picture precisely in case (1).  Assume that $V \rightarrow X$ is a complex
vector bundle of rank $n$ over a smooth complex variety $X$, equipped with a flag of subbundles $F_{\bullet}$, and
a splitting as $V = V' \oplus V''$, with $V'$ and $V''$ being rank $p$ and $q$ subbundles, respectively.  Following
Theorem \ref{thm:orbit-closures}, let $\gamma$ be a $(p,q)$-clan, and over a point $x \in X$, let us say that the
flag $F_{\bullet}(x)$ and the splitting $V'(x) \oplus V''(x)$ are
\textbf{in relative position $\gamma$} if and only if
\begin{enumerate}
	\item $\dim(F_i(x) \cap V'(x)) \geq \gamma(i; +)$
	\item $\dim(F_i(x) \cap V''(x)) \geq \gamma(i; -)$
	\item $\dim(\pi_x(F_i(x)) + F_j(x)) \leq j + \gamma(i; j)$
\end{enumerate}
for all suitable $i,j$.  (In this context, $\pi_x: V(x) \rightarrow V'(x)$ denotes the projection onto the
$p$-dimensional subspace.)  Define
\[ D_{\gamma} := \{ x \in X \ \vert \ F_{\bullet}(x) \text{ and } V'(x) \oplus V''(x) \text{ are in relative position $\gamma$}\}. \]

Then $D_{\gamma}$ is a degeneracy locus which is the set-theoretic preimage of $E \times^K Y_{\gamma}$ under the map
$\phi$, in the notation of Section \ref{sec:deg-loci-general}.  The proof of this, involving standard structures on
the universal spaces, is essentially tautological, so we omit it.

We now describe how a formula for the equivariant class $[Y_{\gamma}]$ implies a formula for the fundamental class
$[D_{\gamma}] \in H^*(X)$ in terms of the Chern classes of $V'$, $V''$, and $F_i/F_{i-1}$ ($i = 1,\hdots,n$).  This
amounts to relating the $S$-equivariant classes in $H_S^*(G/B)$ represented by our $\xx$ and $\yy$-variables
to these Chern classes.

As recalled in Section \ref{sec:deg-loci-general}, the data of the flag $F_{\bullet}$ and the splitting of $V$ as
$V' \oplus V''$ gives us a map
\[ X \stackrel{\phi}{\longrightarrow} BK \times_{BG} BB. \]

It is explained in \cite{Wyser-13-TG} that the classes $x_i$ pull back through this map to $c_1(F_i/F_{i-1})$.  Essentially,
this is because the $x_i$ are the first Chern classes of the standard line bundles on $(G/B)_S = E \times^S (G/B)$,
while the subquotients $F_i/F_{i-1}$ are pullbacks of those same bundles to $X$.  Thus when translating our equivariant
formulas to a formula in $H^*(X)$, $x_i$ is interpreted as $c_1(F_i/F_{i-1})$.

Now, consider the $\yy$-variables.  We claim that the elementary symmetric polynomial $e_i(y_1,\hdots,y_p)$
($i=1,\hdots,p$) is identified with the Chern class $c_i(V')$, while the elementary symmetric
polynomial $e_i(y_{p+1},\hdots,y_n)$ ($i=1,\hdots,q$) is identified with the Chern class $c_i(V'')$.  We
sketch the argument as to why.

The universal space $(G/B)_K = E \times^K G/B$ carries two tautological bundles
$S'_K$ and $S''_K$ of ranks $p$ and $q$, respectively.  Explicitly, the bundle $S'_K$ is 
$(E \times^K \C \left\langle e_1,\hdots,e_p \right\rangle) \times G/B$, while the bundle $S''_K$ is 
$(E \times^K \C \left\langle e_{p+1},\hdots,e_n \right\rangle) \times G/B$.  When pulled back to 
$(G/B)_S$ via the natural map $(G/B)_S \rightarrow (G/B)_K$, both bundles split as direct sums of
line bundles.  $S'_K$ splits as a direct sum of $(E \times^S \C_{Y_i}) \times G/B$ for $i=1,\hdots,p$,
while $S''_K$ splits as a direct sum of $(E \times^S \C_{Y_i}) \times G/B$ for $i=p+1,\hdots,n$.  The
classes $y_i \in H_S^*(G/B)$ are precisely the first Chern classes of these line bundles.  So when we consider
$H_K^*(G/B)$ as a subring of $H_S^*(G/B)$, the Chern classes $c_1(S'_K),\hdots,c_p(S'_K)$ are identically
$e_1(y_1,\hdots,y_p),\hdots,e_p(y_1,\hdots,y_p)$, while the Chern classes $c_1(S''_K),\hdots,c_q(S''_K)$
are $e_1(y_{p+1},\hdots,y_n),\hdots,e_q(y_{p+1},\hdots,y_n)$.  Finally, $V'$ and $V''$ are, respectively, the pullbacks
of $S'_K$ and $S''_K$ through $\phi$.

We give an example.  Suppose we have a smooth complex variety $X$ and a rank $4$ vector bundle $V \rightarrow X$.
Suppose that $V$ splits as a direct sum of rank $2$ subbundles ($V = V' \oplus V''$), and suppose further that $V$ is
equipped with a complete flag of subbundles ($F_1 \subset F_2 \subset F_3 \subset V$).  Let $z_1,z_2,z_3,z_4$ be
$c_1(V')$, $c_2(V')$, $c_1(V'')$, $c_2(V'')$, respectively.  Let $x_i = c_1(F_i/F_{i-1})$ for $i = 1,2,3,4$.  For any
$(2,2)$-clan $\gamma$, we can use the data of Table \ref{tab:type-a-2-2} to determine Chern class formulas for the
class of any locus $D_{\gamma}$ in terms of the $z_i$ and $x_i$.

For instance, consider the clan $\gamma = ++--$.  The formula for $[Y_{\gamma}]$ (found in Table \ref{tab:type-a-2-2}
of the Appendix), when partially expanded and regrouped conveniently, gives
\[ (x_1^2 - x_1(y_3+y_4) + (y_3y_4))(x_2^2 - x_2(y_3+y_4) + (y_3y_4)). \]
We have explained that $y_3 + y_4$ is associated to $z_3$, while $y_3y_4$ is identified to $z_4$.  Thus the
conclusion is that
\[ [D_{++--}] = (x_1^2 - x_1z_3 + z_4)(x_2^2 - x_2z_3 + z_4). \]

\subsection{The degeneracy loci picture in cases (2)-(7)}\label{sec:deg-loci-other-cases}
In the other types, we should have a similar degeneracy locus story.  The vector bundles in the other types
carry an additional structure by virtue of being associated to principal $BSO(n,\C)$ or $BSp(2n,\C)$ bundles,
rather than just principal $BGL(n,\C)$ bundles.  Namely, they carry a non-degenerate quadratic form (types
$B$ and $D$) or skew-symmetric form (type $C$) taking values in the trivial line bundle, and (in types $BD$)
a trivialization of the determinant line bundle.  A lift
of the classifying map to $BB$ for these groups now amounts to a flag which is isotropic or
Lagrangian with respect to this form.

Moreover, since in each of these cases, $K$ is of the form $G \cap K'$ for some $K' = S(GL(p,\C) \times GL(q,\C))$, a
reduction of the structure group of the given bundle to $K$ clearly implies a reduction of
structure group to the corresponding $K'$.  This implies a splitting of the bundle into direct summands of the
appropriate ranks, as we have noted.  A further reduction of structure group to $K$ implies that this splitting has
some further property with respect to the form.  It is easy to see that these additional properties are that the
restriction of the form to each summand is non-degenerate in cases (2), (3), (5), and (7), and that the two rank $n$
summands are orthogonal complements with respect to the form in cases (4) and (6).

So given such a setup, as in the previous subsection, we can see that certain degeneracy loci
are parametrized by the $K$-orbit closures, and that our equivariant formulas for the orbit closures imply
Chern class formulas for the classes of such loci, with the $x$ and $y$-variables interpreted similarly.

Of course, a precise set-theoretic description of the degeneracy loci so parametrized by $K$-orbit
closures depends upon knowing a set-theoretic description of the orbit closures, which we
have only conjectured in types $BC$, and which we have no guess for in any of the type $D$ cases.  Presuming
Conjecture \ref{conj:bruhat-order-other-types} is correct, then the degeneracy loci for the type $B$ and $C$ cases
are described set-theoretically just as those in type $A$ are, with regard to the relative position of the flag and
the splitting; we simply assume the further structures on the bundle to be in place.

In the type $D$ cases, it's not clear precisely what degeneracy loci are being parametrized in general.
We remark that \textit{some} of the orbit closures in the type $D$ cases \textit{are} described set-theoretically
just as in the type $A$ case, so that some of the degeneracy loci in question are just as in the type $A$ case.
However, as we have noted, for other orbit closures, this description is wrong.  Degeneracy loci corresponding
to such orbit closures are described by the type $A$ conditions, plus some additional ones, and it is not clear
what the additional conditions are.

\appendix

\section*{Appendix:  Weak Order Graphs and Tables of Formulas in Examples}

\begin{figure}[h!]
	\caption{$(GL(4,\C),GL(2,\C) \times GL(2,\C))$}\label{fig:type-a-2-2}
	\centering
	\includegraphics[scale=0.5]{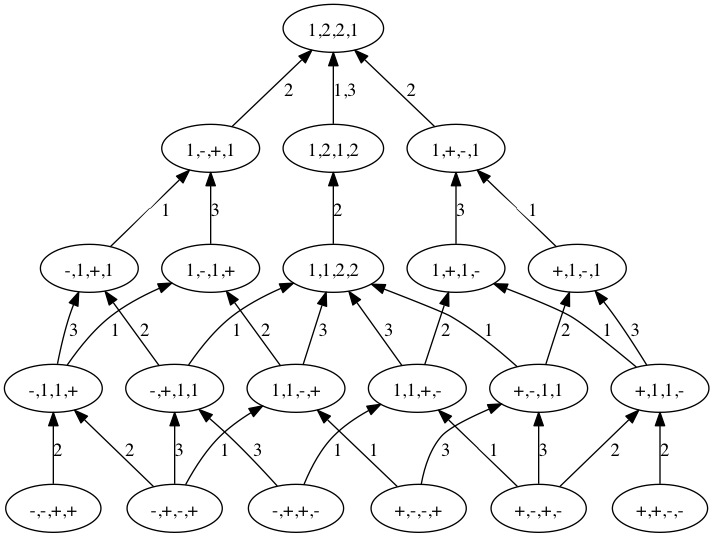}
\end{figure}

\begin{figure}[h!]
	\caption{$(SO(7,\C),S(O(4,\C) \times O(3,\C)))$}\label{fig:type-b-graph}
	\centering
	\includegraphics[scale=0.5]{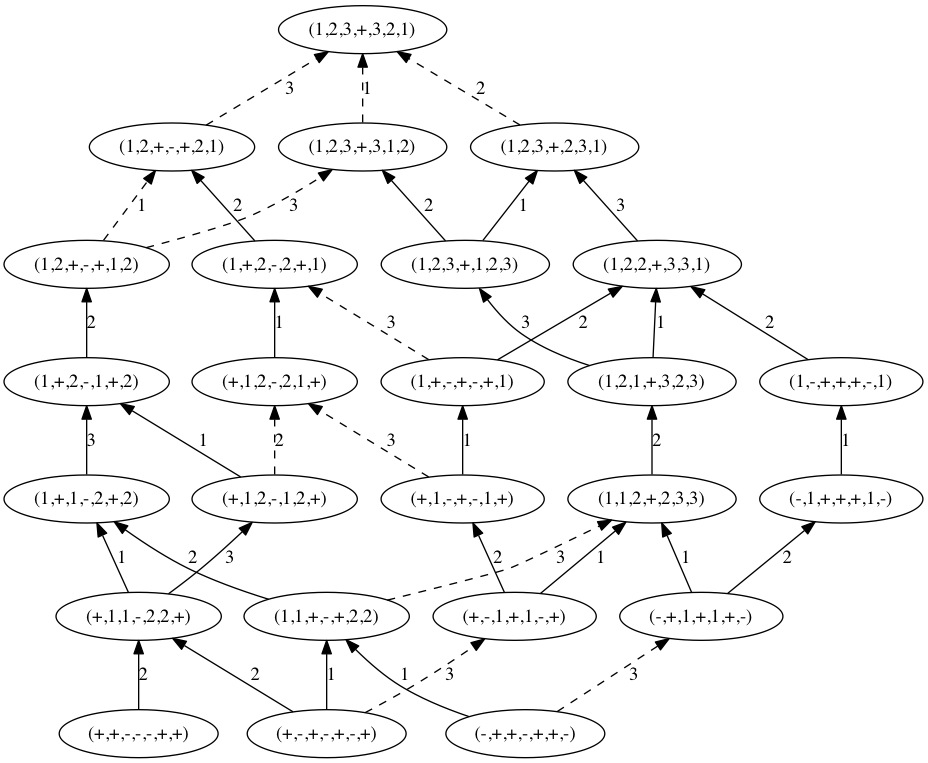}
\end{figure}

\begin{figure}[h!]
	\caption{$(Sp(6,\C),Sp(4,\C) \times Sp(2,\C))$}\label{fig:type-c-graph-1}
	\centering
	\includegraphics[scale=0.5]{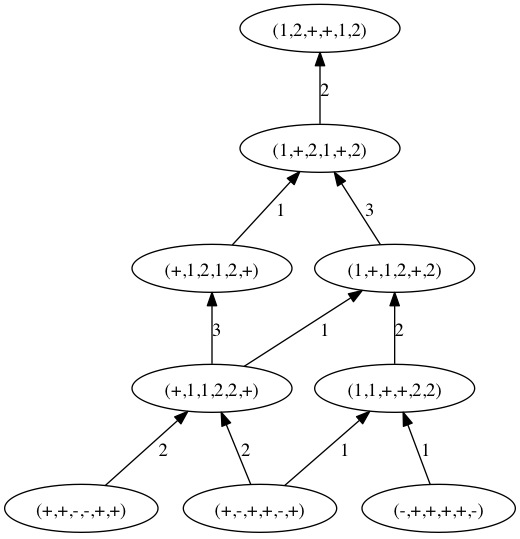}
\end{figure}

\begin{figure}[h!]
	\caption{$(Sp(4,\C),GL(2,\C))$}\label{fig:type-c-graph-2}
	\centering
	\includegraphics[scale=0.5]{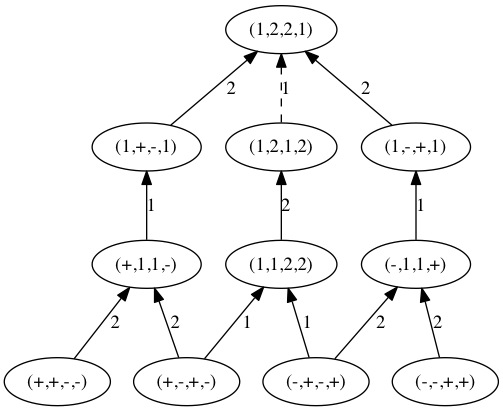}
\end{figure}

\begin{figure}[h!]
	\caption{$(SO(6,\C),S(O(4,\C) \times O(2,\C)))$}\label{fig:type-d-graph-1}
	\centering
	\includegraphics[scale=0.5]{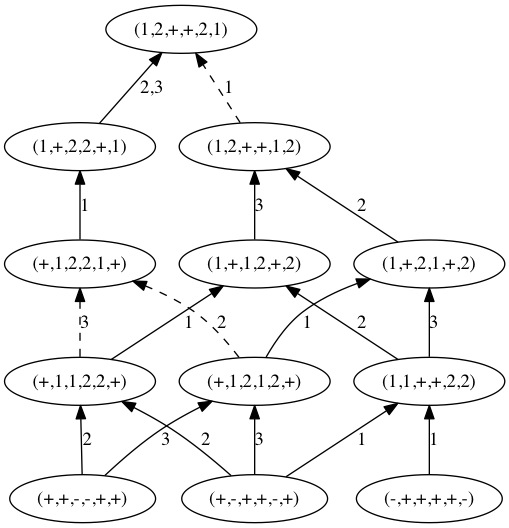}
\end{figure}

\begin{figure}[h!]
	\caption{$(SO(6,\C),GL(3,\C))$}\label{fig:type-d-graph-2}
	\centering
	\includegraphics[scale=0.5]{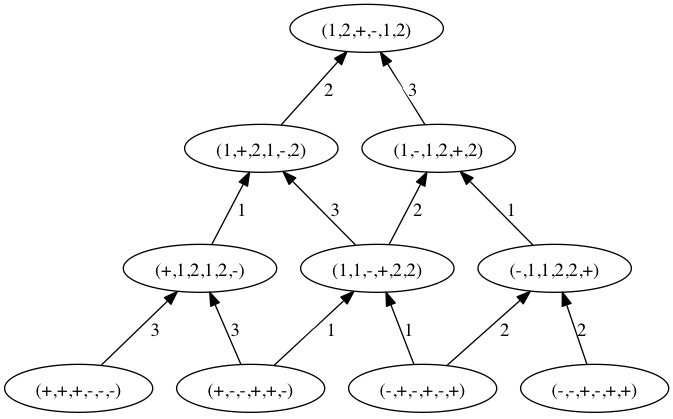}
\end{figure}

\begin{figure}[h!]
	\caption{$(SO(6,\C),S(O(3,\C) \times O(3,\C)))$}\label{fig:type-d-graph-3}
	\centering
	\includegraphics[scale=0.5]{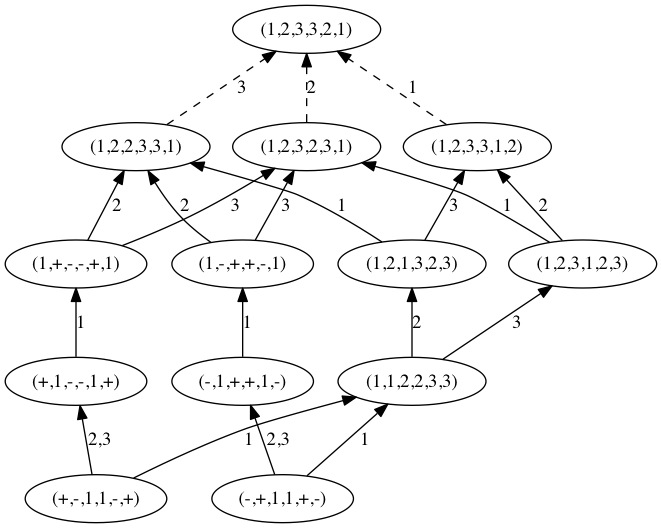}
\end{figure}

\begin{table}[h]
	\caption{Formulas for $(GL(4,\C),GL(2,\C) \times GL(2,\C))$}\label{tab:type-a-2-2}
	\resizebox{16cm}{!}{	
		\begin{tabular}{|l|l|}
			\hline
			$(2,2)$-clan $\gamma$ & Formula for $[Y_{\gamma}]$ \\ \hline
			$++--$ & $(x_1-y_3)(x_1-y_4)(x_2-y_3)(x_2-y_4)$ \\ \hline
			$+-+-$ & $-(x_1-y_3)(x_1-y_4)(x_3-y_3)(x_3-y_4)$ \\ \hline
			$+--+$ & $(x_1-y_3)(x_1-y_4)(x_4-y_3)(x_4-y_4)$ \\ \hline
			$-++-$ & $(x_2-y_3)(x_2-y_4)(x_3-y_3)(x_3-y_4)$ \\ \hline
			$-+-+$ & $-(x_2-y_3)(x_2-y_4)(x_4-y_3)(x_4-y_4)$ \\ \hline
			$--++$ & $(x_3-y_3)(x_3-y_4)(x_4-y_3)(x_4-y_4)$ \\ \hline
			$+11-$ & $(x_1-y_3)(x_1-y_4)(x_2+x_3-y_3-y_4)$ \\ \hline
			$11+-$ & $-(x_3-y_3)(x_3-y_4)(x_1+x_2-y_3-y_4)$ \\ \hline
			$+-11$ & $-(x_1-y_3)(x_1-y_4)(x_3+x_4-y_3-y_4)$ \\ \hline
			$11-+$ & $(x_4-y_3)(x_4-y_4)(x_1+x_2-y_3-y_4)$ \\ \hline
			$-+11$ & $(x_2-y_3)(x_2-y_4)(x_3+x_4-y_3-y_4)$ \\ \hline
			$-11+$ & $-(x_4-y_3)(x_4-y_4)(x_2+x_3-y_3-y_4)$ \\ \hline
			$1+1-$ & $x_1x_2+x_1x_3+x_2x_3-x_1y_3-x_2y_3-x_3y_3+y_3^2-x_1y_4-x_2y_4-x_3y_4+y_3y_4+y_4^2$ \\ \hline
			$+1-1$ & $(x_1-y_3)(x_1-y_4)$ \\ \hline
			$1122$ & $-(x_1+x_2-y_3-y_4)(x_3+x_4-y_3-y_4)$ \\ \hline
			$1-1+$ & $(x_4-y_3)(x_4-y_4)$ \\ \hline
			$-1+1$ & $x_2x_3+x_2x_4+x_3x_4-x_2y_3-x_3y_3-x_4y_3+y_3^2-x_2y_4-x_3y_4-x_4y_4+y_3y_4+y_4^2$ \\ \hline
			$1+-1$ & $x_1+x_2-y_3-y_4$ \\ \hline
			$1212$ & $x_1-x_4$ \\ \hline
			$1-+1$ & $-x_3-x_4+y_3+y_4$ \\ \hline
			$1221$ & $1$ \\
			\hline
		\end{tabular}
	}
\end{table}

\begin{table}[h]
	\caption{Formulas for $(SO(7,\C),S(O(4,\C) \times O(3,\C)))$}\label{tab:type-b-table}
	\begin{tabular}{|c|l|}
		\hline
		Symmetric $(4,3)$-clan $\gamma$ & Formula for $[Y_{\gamma}]$ \\ \hline
		$++---++$ & $x_1x_2(x_1-y_3)(x_1+y_3)(x_2-y_3)(x_2+y_3)$ \\ \hline
		$+-+-+-+$ & $-x_1x_3(x_1-y_3)(x_1+y_3)(x_3-y_3)(x_3+y_3)$ \\ \hline
		$-++-++-$ & $x_2x_3(x_2-y_3)(x_2+y_3)(x_3-y_3)(x_3+y_3)$ \\ \hline
		$+11-22+$ & $x_1(x_1-y_3)(x_1+y_3)(x_2^2+x_2x_3+x_3^2-y_3^2)$  \\ \hline		
		$11+-+22$ & $-x_3(x_3-y_3)(x_3+y_3)(x_1^2+x_1x_2+x_2^2-y_3^2)$ \\ \hline
		$+-1+1-+$ & $-x_1(x_1-y_3)(x_1+y_3)(x_3-y_3)(x_3+y_3)$ \\ \hline
		$-+1+1+-$ & $x_2(x_2-y_3)(x_2+y_3)(x_3-y_3)(x_3+y_3)$ \\ \hline
		$1+1-2+2$ & $x_1^2x_2^2+x_1^2x_2x_3+x_1x_2^2x_3+x_1^2x_3^2+x_1x_2x_3^2+x_2^2x_3^2-x_1^2y_3^2-x_2^2y_3^2-x_3^2y_3^2+y_3^4$ \\ \hline
		$+12-12+$ & $2x_1x_2(x_1-y_3)(x_1+y_3)$ \\ \hline		
		$+1-+-1+$ & $-x_1(x_1-y_3)(x_1+y_3)(x_2+x_3)$ \\ \hline
		$113+322$ & $-(x_3-y_3)(x_3+y_3)(x_1^2+x_1x_2+x_2^2-y_3^2)$ \\ \hline
		$-1+++1-$ & $(x_2-y_3)(x_2+y_3)(x_3-y_3)(x_3+y_3)$ \\ \hline
		$1+2-1+2$ & $2x_1x_2(x_1+x_2)$ \\ \hline
		$+12-21+$ & $x_1(x_1-y_3)(x_1+y_3)$ \\ \hline
		$1+-+-+1)$ & $x_1^2x_2+x_1x_2^2+x_1^2x_3+x_1x_2x_3+x_2^2x_3-x_3y_3^2$ \\ \hline
		$131+232$ & $x_1(x_1x_2+x_1x_3+x_2x_3+y_3^2)$ \\ \hline
		$1-+++-1$ & $-(x_1+x_2)(x_3-y_3)(x_3+y_3)$ \\ \hline
		$12+-+12$ & $2x_1(x_1+x_2+x_3)$ \\ \hline
		$1+2-2+1$ & $x_1^2+x_1x_2+x_2^2-y_3^2$ \\ \hline
		$132+132$ & $2x_1(x_1+x_2)$ \\ \hline
		$311+223$ & $x_1x_2+x_1x_3+x_2x_3+y_3^2$ \\ \hline
		$12+-+21$ & $x_1+x_2+x_3$ \\ \hline
		$123+312$ & $2x_1$ \\ \hline
		$312+123$ & $2(x_1+x_2)$ \\ \hline
		$123+321$ & $1$ \\
		\hline
	\end{tabular}
\end{table}

\begin{table}[h]
	\caption{Formulas for $(Sp(6,\C),Sp(4,\C) \times Sp(2,\C))$}\label{tab:type-c-table-1}
	\begin{tabular}{|l|l|}
		\hline
		Symmetric $(4,2)$-clan $\gamma$ & Formula for $[Y_{\gamma}]$ \\ \hline
		$++--++$ & $(x_1-y_3)(x_1+y_3)(x_2-y_3)(x_2+y_3)$ \\ \hline
		$+-++-+$ & $-(x_1-y_3)(x_1+y_3)(x_3-y_3)(x_3+y_3)$ \\ \hline
		$-++++-$ & $(x_2-y_3)(x_2+y_3)(x_3-y_3)(x_3+y_3)$ \\ \hline
		$+1122+$ & $(x_1-y_3)(x_1+y_3)(x_2+x_3)$ \\ \hline
		$11++22$ & $-(x_3-y_3)(x_3+y_3)(x_1+x_2)$ \\ \hline
		$1+12+2$ & $x_1x_2+x_1x_3+x_2x_3+y_3^2$ \\ \hline
		$+1212+$ & $(x_1-y_3)(x_1+y_3)$ \\ \hline
		$1+21+2$ & $x_1+x_2$ \\ \hline
		$12++12$ & $1$ \\
		\hline
	\end{tabular}
\end{table}

\begin{table}[h]
	\caption{Formulas for $(Sp(4,\C),GL(2,\C))$}\label{tab:type-c-table-2}
	\begin{tabular}{|l|l|}
		\hline
		Skew-symmetric $(2,2)$-clan $\gamma$ & Formula for $[Y_{\gamma}]$ \\ \hline
		$++--$ & $(x_1+x_2+y_1+y_2)(x_1x_2+y_1y_2)$ \\ \hline
		$+-+-$ & $-(x_1-x_2+y_1+y_2)(-x_1x_2+y_1y_2)$ \\ \hline
		$-+-+$ & $(-x_1+x_2+y_1+y_2)(-x_1x_2+y_1y_2)$ \\ \hline
		$--++$ & $-(-x_1-x_2+y_1+y_2)(x_1x_2+y_1y_2))$ \\ \hline
		$+11-$ & $(x_1+y_1)(x_1+y_2)$ \\ \hline
		$1122$ & $2(x_1x_2-y_1y_2)$ \\ \hline
		$-11+$ & $(x_1-y_1)(x_1-y_2)$ \\ \hline
		$1+-1$ & $x_1+x_2+y_1+y_2$ \\ \hline
		$1212$ & $2x_1$ \\ \hline
		$1-+1$ & $x_1+x_2-y_1-y_2$ \\ \hline
		$1221$ & $1$ \\
		\hline
	\end{tabular}
\end{table}

\begin{table}[h]
	\caption{Formulas for $(SO(6,\C),S(O(4,\C) \times O(2,\C)))$}\label{tab:type-d-table-1}
	\begin{tabular}{|l|l|}
		\hline
		Symmetric $(4,2)$-clan $\gamma$ & Formula for $[Y_{\gamma}]$ \\ \hline
		$++--++$ & $(x_1-y_3)(x_1+y_3)(x_2-y_3)(x_2+y_3)$ \\ \hline
		$+-++-+$ & $-(x_1-y_3)(x_1+y_3)(x_3-y_3)(x_3+y_3)$ \\ \hline
		$-++++-$ & $(x_2-y_3)(x_2+y_3)(x_3-y_3)(x_3+y_3)$ \\ \hline
		$+1122+$ & $(x_1-y_3)(x_1+y_3)(x_2+x_3)$ \\ \hline
		$+1212+$ & $(x_1-y_3)(x_1+y_3)(x_2-x_3)$ \\ \hline
		$11++22$ & $-(x_3-y_3)(x_3+y_3)(x_1+x_2)$ \\ \hline
		$1+12+2$ & $x_1x_2+x_1x_3+x_2x_3+y_3^2$ \\ \hline
		$+1221+$ & $(x_1-y_3)(x_1+y_3)$ \\ \hline
		$1+21+2$ & $x_1x_2-x_1x_3-x_2x_3+y_3^2$ \\ \hline
		$12++12$ & $2x_1$ \\ \hline
		$1+22+1$ & $x_1+x_2$ \\ \hline
		$12++21$ & $1$ \\ 
		\hline
	\end{tabular}
\end{table}

\begin{table}[h]
	\caption{Formulas for $(SO(6,\C),GL(3,\C))$}\label{tab:type-d-table-2}
	\resizebox{18cm}{3cm}{
		\begin{tabular}{|l|l|}
			\hline
			Skew-symmetric $(3,3)$-clan $\gamma$ & Formula for $[Y_{\gamma}]$ \\ \hline
			$+++---$ & $\frac{1}{4}\Delta_2(\xx,\yy,id)$ \\ \hline
			$--+-++$ & $-\frac{1}{4}\Delta_2(\xx,\yy,(\overline{1}\ \overline{2}\ 3))$ \\ \hline
			$-+-+-+$ & $\frac{1}{4}\Delta_2(\xx,\yy,(\overline{1}\ 2\ \overline{3}))$ \\ \hline
			$+--++-$ & $-\frac{1}{4}\Delta_2(\xx,\yy,(1\ \overline{2}\ \overline{3}))$ \\ \hline
			$+1212-$ & $\frac{1}{2}(x_1^2+x_2x_3+x_1y_1+x_1y_2+y_1y_2+x_1y_3+y_1y_3+y_2y_3)$ \\ \hline
			$-1122+$ & $\frac{1}{2}(x_1^2-x_2x_3-x_1y_1-x_1y_2+y_1y_2-x_1y_3+y_1y_3+y_2y_3)$ \\ \hline
			$11-+22$ & $\frac{1}{2}(x_1x_2-x_3^2+x_3y_1+x_3y_2-y_1y_2+x_3y_3-y_1y_3-y_2y_3)$ \\ \hline
			$1+21-2$ & $\frac{1}{2}(x_1+x_2-x_3+y_1+y_2+y_3)$ \\ \hline
			$1-12+2$ & $\frac{1}{2}(x_1+x_2+x_3-y_1-y_2-y_3)$ \\ \hline
			$12+-12$ & $1$ \\
			\hline
		\end{tabular}
	}
\end{table}

\begin{table}[h]
	\caption{Formulas for $(SO(6,\C),S(O(3,\C) \times O(3,\C)))$}\label{tab:type-d-table-3}
	\begin{tabular}{|l|l|}
		\hline
		Symmetric $(3,3)$-clan $\gamma$ & Formula for $[Y_{\gamma}]$ \\ \hline
		$+-11-+$ &  $x_1x_2(x_1-y_3)(x_1+y_3)$ \\ \hline
		$-+11+-$ & $-x_1x_2(x_2-y_3)(x_2+y_3)$ \\ \hline
		$112233$ & $x_1x_2(x_1+x_2)$ \\ \hline
		$+1--1+$ & $x_1(x_1-y_3)(x_1+y_3)$ \\ \hline
		$-1++1-$ & $-x_1(x_2^2+x_2x_3+x_3^2-y_3^2)$ \\ \hline
		$121323$ & $x_1(x_1+x_2+x_3)$ \\ \hline
		$123123$ & $x_1(x_1+x_2-x_3)$ \\ \hline
		$1+--+1$ & $x_1^2+x_1x_2+x_2^2-y_3^2$ \\ \hline
		$1-++-1$ & $x_1x_2-x_3^2+y_3^2$ \\ \hline
		$122331$ & $x_1+x_2+x_3$ \\ \hline
		$123312$ & $2x_1$ \\ \hline
		$123231$ & $x_1+x_2-x_3$ \\ \hline
		$123321$ & $1$ \\
		\hline
	\end{tabular}
\end{table}

\bibliographystyle{alpha}
\bibliography{../../sourceDatabase}

\def\cprime{$'$}
\begin{thebibliography}{{Wys}13b}

\bibitem[AP]{Achinger-Perrin}
Piotr Achinger and Nicolas Perrin.
\newblock Spherical multiple flags.
\newblock {\em arXiv:1307.7236}.

\bibitem[Bri98]{Brion-98_i}
Michel Brion.
\newblock Equivariant cohomology and equivariant intersection theory.
\newblock In {\em Representation theories and algebraic geometry ({M}ontreal,
  {PQ}, 1997)}, volume 514 of {\em NATO Adv. Sci. Inst. Ser. C Math. Phys.
  Sci.}, pages 1--37. Kluwer Acad. Publ., Dordrecht, 1998.
\newblock Notes by Alvaro Rittatore.

\bibitem[Bri99]{Brion-99}
M.~Brion.
\newblock Rational smoothness and fixed points of torus actions.
\newblock {\em Transform. Groups}, 4(2-3):127--156, 1999.
\newblock Dedicated to the memory of Claude Chevalley.

\bibitem[Bri01]{Brion-01}
Michel Brion.
\newblock On orbit closures of spherical subgroups in flag varieties.
\newblock {\em Comment. Math. Helv.}, 76(2):263--299, 2001.

\bibitem[Bri03]{Brion-03}
Michel Brion.
\newblock Multiplicity-free subvarieties of flag varieties.
\newblock In {\em Commutative algebra ({G}renoble/{L}yon, 2001)}, volume 331 of
  {\em Contemp. Math.}, pages 13--23. Amer. Math. Soc., Providence, RI, 2003.

\bibitem[Ful92]{Fulton-92}
William Fulton.
\newblock Flags, {S}chubert polynomials, degeneracy loci, and determinantal
  formulas.
\newblock {\em Duke Math. J.}, 65(3):381--420, 1992.

\bibitem[Ful96a]{Fulton-96_2}
William Fulton.
\newblock Determinantal formulas for orthogonal and symplectic degeneracy loci.
\newblock {\em J. Differential Geom.}, 43(2):276--290, 1996.

\bibitem[Ful96b]{Fulton-96_1}
William Fulton.
\newblock Schubert varieties in flag bundles for the classical groups.
\newblock In {\em Proceedings of the {H}irzebruch 65 {C}onference on
  {A}lgebraic {G}eometry ({R}amat {G}an, 1993)}, volume~9 of {\em Israel Math.
  Conf. Proc.}, pages 241--262, Ramat Gan, 1996. Bar-Ilan Univ.

\bibitem[Ful97]{Fulton-YoungTableaux}
William Fulton.
\newblock {\em Young tableaux}, volume~35 of {\em London Mathematical Society
  Student Texts}.
\newblock Cambridge University Press, Cambridge, 1997.
\newblock With applications to representation theory and geometry.

\bibitem[Gra97]{Graham-97}
William Graham.
\newblock The class of the diagonal in flag bundles.
\newblock {\em J. Differential Geom.}, 45(3):471--487, 1997.

\bibitem[IMN11]{Ikeda-Mihalcea-Naruse}
Takeshi Ikeda, Leonardo~C. Mihalcea, and Hiroshi Naruse.
\newblock Double {S}chubert polynomials for the classical groups.
\newblock {\em Adv. Math.}, 226(1):840--886, 2011.

\bibitem[KT02]{Kresch-Tamvakis}
A.~Kresch and H.~Tamvakis.
\newblock Double {S}chubert polynomials and degeneracy loci for the classical
  groups.
\newblock {\em Ann. Inst. Fourier (Grenoble)}, 52(6):1681--1727, 2002.

\bibitem[Mat79]{Matsuki-79}
Toshihiko Matsuki.
\newblock The orbits of affine symmetric spaces under the action of minimal
  parabolic subgroups.
\newblock {\em J. Math. Soc. Japan}, 31(2):331--357, 1979.

\bibitem[M{\=O}90]{Matsuki-Oshima-90}
Toshihiko Matsuki and Toshio {\=O}shima.
\newblock Embeddings of discrete series into principal series.
\newblock In {\em The orbit method in representation theory ({C}openhagen,
  1988)}, volume~82 of {\em Progr. Math.}, pages 147--175. Birkh\"auser Boston,
  Boston, MA, 1990.

\bibitem[PR97]{Pragacz-Ratajski}
P.~Pragacz and J.~Ratajski.
\newblock Formulas for {L}agrangian and orthogonal degeneracy loci;
  {$\widetilde{Q}$}-polynomial approach.
\newblock {\em Compositio Math.}, 107(1):11--87, 1997.

\bibitem[Pra96]{Pragacz}
Piotr Pragacz.
\newblock Symmetric polynomials and divided differences in formulas of
  intersection theory.
\newblock In {\em Parameter spaces ({W}arsaw, 1994)}, volume~36 of {\em Banach
  Center Publ.}, pages 125--177. Polish Acad. Sci., Warsaw, 1996.

\bibitem[RS90]{Richardson-Springer-90}
R.~W. Richardson and T.~A. Springer.
\newblock The {B}ruhat order on symmetric varieties.
\newblock {\em Geom. Dedicata}, 35(1-3):389--436, 1990.

\bibitem[RS93]{Richardson-Springer-92}
R.~W. Richardson and T.~A. Springer.
\newblock Combinatorics and geometry of {$K$}-orbits on the flag manifold.
\newblock In {\em Linear algebraic groups and their representations ({L}os
  {A}ngeles, {CA}, 1992)}, volume 153 of {\em Contemp. Math.}, pages 109--142.
  Amer. Math. Soc., Providence, RI, 1993.

\bibitem[Smi07]{Smirnov-Thesis}
E.~Yu. Smirnov.
\newblock {\em Orbites d'un sous-groupe de Borel dans le produit de deux
  grassmanniennes}.
\newblock PhD thesis, l'Universit\'e Joseph Fourier - Grenoble I, 2007.

\bibitem[Smi08]{Smirnov-08}
E.~Yu. Smirnov.
\newblock Resolutions of singularities for {S}chubert varieties in double
  {G}rassmannians.
\newblock {\em Funktsional. Anal. i Prilozhen.}, 42(2):56--67, 96, 2008.

\bibitem[Spr85]{Springer-85}
T.~A. Springer.
\newblock Some results on algebraic groups with involutions.
\newblock In {\em Algebraic groups and related topics ({K}yoto/{N}agoya,
  1983)}, volume~6 of {\em Adv. Stud. Pure Math.}, pages 525--543.
  North-Holland, Amsterdam, 1985.

\bibitem[Vog83]{Vogan-83}
David~A. Vogan.
\newblock Irreducible characters of semisimple {L}ie groups. {III}. {P}roof of
  {K}azhdan-{L}usztig conjecture in the integral case.
\newblock {\em Invent. Math.}, 71(2):381--417, 1983.

\bibitem[WWar]{Woo-Wyser-14}
Alexander Woo and {B}enjamin~{J.} Wyser.
\newblock Combinatorial results on {$(1,2,1,2)$}-avoiding {$GL(p,\C) \times
  GL(q,\C)$}-orbit closures on {$GL(p+q,\C)/B$}.
\newblock {\em Int. Math. Res. Not.}, to appear.

\bibitem[WY14]{Wyser-Yong-13}
{B}enjamin~{J.} Wyser and {A}lexander Yong.
\newblock Polynomials for {$GL_p \times GL_q$} orbit closures in the flag
  variety.
\newblock {\em Selecta Math. (N.S.)}, 20(4):1083--1110, 2014.

\bibitem[Wys12a]{Wyser-11b}
Benjamin~J. Wyser.
\newblock {$K$}-orbits on {$G/B$} and {S}chubert constants for pairs of signed
  shuffles in types {$C$} and {$D$}.
\newblock {\em J. Algebra}, 364:67--87, 2012.

\bibitem[{Wys}12b]{Wyser-Thesis}
Benjamin~J. {Wyser}.
\newblock {\em Symmetric subgroup orbit closures on flag varieties: Their
  equivariant geometry, combinatorics, and connections with degeneracy loci}.
\newblock PhD thesis, University of Georgia, 2012.

\bibitem[Wys13a]{Wyser-13-TG}
Benjamin~J. Wyser.
\newblock K-orbit closures on {G}/{B} as universal degeneracy loci for flagged
  vector bundles with symmetric or skew-symmetric bilinear form.
\newblock {\em Transform. Groups}, 18(2):557--594, 2013.

\bibitem[{Wys}13b]{Wyser-12b}
Benjamin~J. {Wyser}.
\newblock {Schubert calculus of Richardson varieties stable under spherical
  Levi subgroups}.
\newblock {\em J. Algebraic Combin.}, 38(4):829--850, 2013.

\bibitem[Yam97]{Yamamoto-97}
Atsuko Yamamoto.
\newblock Orbits in the flag variety and images of the moment map for classical
  groups. {I}.
\newblock {\em Represent. Theory}, 1:329--404 (electronic), 1997.

\end{thebibliography}

\end{document}